\documentclass{amsart}
\usepackage{titletoc}
\usepackage{graphicx}
\usepackage{amsmath}
\usepackage{amssymb}
\usepackage{color}
\usepackage{mathdots}
\usepackage{optidef}
\usepackage{amsfonts}
\usepackage{mathrsfs}
\usepackage{hyperref}
\usepackage{kbordermatrix}
\usepackage{mathtools}

\contentsmargin{0pt}

\newtheorem{thm}{Theorem}[section]
\newtheorem{cor}[thm]{Corollary}

\newtheorem{lem}[thm]{Lemma}

\theoremstyle{definition}
\newtheorem{defn}[thm]{Definition}
\theoremstyle{remark}
\newtheorem{rem}[thm]{Remark}
\theoremstyle{definition}

\newtheorem{example}[thm]{Example}

\numberwithin{equation}{section}

\newcommand{\sA}{\mathscr{A}}

\newcommand{\cI}{\mathcal{I}}

\newcommand{\cA}{\mathcal{A}}
\newcommand{\cB}{\mathcal{B}}
\newcommand{\cC}{\mathcal{C}}

\newcommand{\cM}{\mathcal{M}}
\newcommand{\cN}{\mathcal{N}}

\newcommand{\cP}{\mathcal{P}}
\newcommand{\cZ}{\mathcal{Z}}

\newcommand{\cV}{\mathcal{V}}

\newcommand{\RR}{\mathbb{R}}

\newcommand{\CC}{\mathbb{C}}
\newcommand{\NN}{\mathbb{N}}

\newcommand{\wt}{\widetilde}

\newcommand{\col}{{\rm col}}
\newcommand{\row}{{\rm row}}
\newcommand{\rank}{{\rm rank \,}}
\newcommand{\card}{{\rm card \,}}
\newcommand{\supp}{{\rm supp \,}}
\renewcommand{\kbldelim}{(}
\renewcommand{\kbrdelim}{)}

\begin{document}
\title[The indefinite truncated multidimensional moment problem]{On a minimal solution for the indefinite truncated multidimensional moment problem} 
\author[DP Kimsey]{David P. Kimsey}
\address{School of Mathematics, Statistics and Physics\\
Newcastle University\\
Newcastle upon Tyne NE1 7RU UK}
\email{david.kimsey@ncl.ac.uk}

\begin{abstract} 

We will consider the indefinite truncated multidimensional moment problem. Necessary and sufficient conditions for a given truncated multisequence to have a signed representing measure $\mu$ with $\card \supp \mu$ as small as possible are given by the existence of a rank preserving extension of a multivariate Hankel matrix (built from the given truncated multisequence) such that the corresponding associated polynomial ideal is real radical. This result is a special case of a more general characterisation of truncated multisequences with a minimal complex representing measure whose support is symmetric with respect to complex conjugation (which we will call {\it quasi-complex}). One motivation for our results is the fact that positive semidefinite truncated multisequence need not have a positive representing measure. Thus, our main result gives the potential for computing a signed representing measure $\mu = \mu_+ - \mu_-$, where $\card \mu_-$ is small. We illustrate this point on concrete examples.



\end{abstract}

\maketitle

\tableofcontents

\section{Introduction}
\label{sec:INTRO}

In this paper, we will investigate the {\it indefinite truncated multidimensional moment problem}. Given a real-valued truncated multisequence $s = (s_{\gamma})_{\stackrel{0 \leq |\gamma| \leq m}{\gamma \in \NN_0^d} }$, we wish to find necessary and sufficient conditions on $s$ for the existence of a signed measure $\mu$ on $\RR^d$, with convergent moments, such that
\begin{equation}
\label{eq:RM}
s_{\gamma} = \int_{\RR^d} x^{\gamma} d\mu(x) := \int \cdots \int_{\RR^d} \prod_{j=1}^d x_j^{\gamma_j} d\mu(x_1, \ldots, x_d)
\end{equation}
for all $0 \leq |\gamma| \leq m$
and
\begin{equation}
\label{eq:MINIMALITY}
\text{$\card \supp \mu$ is as small as possible.}
\end{equation} 
If \eqref{eq:RM} holds, then $\mu$ is called a {\it signed representing measure} for $s$. If \eqref{eq:RM} and \eqref{eq:MINIMALITY} are in force, then $\mu$ is called {\it minimal signed representing measure for $s$}. We will also be interested in the case when the representing measure $\mu$ is supported on $\CC^d$ with the requirement that
$$z \in \supp \mu \Longleftrightarrow \bar{z} \in \supp \mu.$$
In this case, \eqref{eq:RM} takes the form
\begin{equation}
\label{eq:QCM}
s_{\gamma} = \int_{\CC^d} z^{\gamma} d\mu(z) := \int \cdots \int_{\CC^d} \prod_{j=1}^d z_j^{\gamma_j} d\mu(z_1, \ldots, z_d)  \quad \quad {\rm for} \quad 0 \leq |\gamma | \leq m.
\end{equation}
In this case, $\mu$ is called a {\it quasi-complex representing measure} for $s$.

It is well-known that a truncated multisequence $s = (s_{\gamma})_{0 \leq |\gamma| \leq m}$ always has a signed representing measure. Indeed, one can always choose points $$w^{(1)}, \ldots, w^{( \ell )} \in \RR^d,$$ 
where
$$\ell := \ell(m,d) := \binom{m+d}{d},$$
such that the Vandermonde matrix $V$ based on $w^{(1)}, \ldots, w^{( \ell )}$ and the indexing set $\{ \lambda \in \NN_0^d: 0 \leq |\lambda | \leq m \}$ (see Definition \ref{def:VM}) is invertible. Consequently, $\varrho_1, \ldots, \varrho_{\ell}$ are given by
\begin{equation}
\label{eq:RHOS}
 {\rm col}(\varrho_a)_{a=1}^{ \ell } = (V^T)^{-1} {\rm col}(s_{\gamma})_{0 \leq |\gamma| \leq m}
\end{equation}
and $\mu = \sum_{a=1}^{\ell} \varrho_a \delta_{w^{(a)}}$ is a signed representing measure for $s$. Thus, the indefinite truncated moment problem is nontrivial if the representing measure $\mu$ satisfies
$$\card \supp \mu < \ell(m,d).$$

We shall see that if we let $M(n)$ be the $d$-Hankel matrix based (see Definition \ref{def:Feb13dH}) on $s = (s_{\gamma})_{0 \leq |\gamma| \leq 2n}$ and $M(n)$ (which is not necessarily positive semidefinite) be the corresponding multivarite Hankel matrix built from $s$, then $s$ has a minimal signed representing measure (i.e., $r$-atomic, where $r = \rank M(n)$) if and only if $M(n)$ has a rank preserving extension $M(n+1)$ such that {\it variety of $M(n+1)$} (see Definition \ref{def:Variety}) satisfies
\begin{equation}
\label{eq:CCs}
\card \cV(M(n+1)) = r
\end{equation}
and
\begin{equation}
\label{eq:Varietycondition}
\cV(M(n+1)) \subseteq \RR^d
\end{equation}
(see Theorem \ref{thm:Feb7m1} for a precise statement). It turns out that \eqref{eq:CCs} is equivalent to checking that the certain shift matrices which are determined by the extension are diagonalisable. It is also worth noting that in the particular case that the extension is positive semidefinite, then \eqref{eq:CCs} and \eqref{eq:Varietycondition} hold automatically and the corresponding minimal signed representing measure will be positive and we recover the celebrated {\it flat extension theorem} of Curto and Fialkow that began in an equivalent complex moment setting in \cite{CF_memoir} (and has been further investigated in \cite{CF_memoir, CF_analytic, CF_RGR, CF_quadratic, CF_quartic, CF_THMP,CF_several, Fialkow_FV, F_CV, FN_H, CF_bivariate, F_Parallel, F_Limits, F_Survey, F_CVariety}). For recent approaches to the truncated multidimensional moment problem which are independent of the flat extension theorem, see, e.g., the {\it core variety} approach of Fialkow ((\cite{Fialkow-CV}) and Blekherman and Fialkow (\cite{blekherman:2018}) and also an LMI type condition given by the author and Putinar \cite{KP}.

We shall also see that if we enlarge the class of representing measures to include complex measures $\mu$ such that
\begin{equation}
\lambda \in \supp \mu \Longleftrightarrow \bar{\lambda} \in \supp \mu, 
\end{equation}
then $s$ has a minimal quasicomplex representing measure if and only if $M(n)$ has a rank preserving extension such that
$$\card \cV(M(n+1)) = r$$
(see Theorem \ref{thm:Feb7m1} for a precise statement).

\begin{rem}
\label{rem:CMP}
In principle, the main results (formulated precisely in Theorem \ref{thm:Feb7m1}), can be formulated and proved in a complex truncated multidimensional moment problem setting. We will not pursue this here. 
\end{rem}

As far as we know, indefinite moment problems have been investigated in a univariate setting with the requirement that the representing measure $\mu$ has a Jordan decomposition $\mu = \mu_+ - \mu_-$ with 
$$\card \supp \mu_- = \kappa_-$$ 
for some given $\kappa_- \geq 0$ (see, e.g., Krein and Langer \cite{KL}, Berg, Christensen and Maverick \cite{BCM} and Derkach, Hassi and de Snoo \cite{DHdS1}, in a trigonometric moment problem setting (see, e.g., Bakonyi and Lopushanskaya \cite{BL} and also Bakonti and Woerdeman's book \cite{BW}) and also an abstract multidimensional semigroup setting (see, e.g., \cite{BS} and the book of Sasv\'ari \cite{Sasvari}). However, the problem of obtaining condition(s) for the existence of a minimal representing measure for a truncated multidimensional moment sequence seems to have been left open for quite some time.

\noindent {\bf Motivation}

Given a truncated multisequence $s = (s_{\gamma})_{0 \leq |\gamma| \leq 2n}$, it is well known that the positivity of the $d$-Hankel matrix $M(n)$ based on $s$ does not imply the existence of a {\it positive} representing measure for $s$ (see, e.g., Example 3.11 in \cite{Fialkow_FV}). However, Theorem \ref{thm:Feb7m1} provides one with a tool for computing a signed representing measure $\mu = \mu_+ - \mu_-$ for $s$, where
$$\text{$\card \mu_-$ is ``small'' relative to $\rank M(n)$.}$$
In this care, $\mu = \mu_+ - \mu_-$ may be an acceptable substitute for a positive representing measure for $s$.

The paper is organised as follows. In Section \ref{sec:MHM}, we will provide basic results for $d$-Hankel matrices. In Section \ref{sec:QCM}, we will introduce a seemingly new class of representing measures which we will call quasi-complex. In Section \ref{sec:Pontryagin}, we will construct a Pontryagin space associated with a multisequence whose infinite $d$-Hankel matrix has finite rank. In Section \ref{sec:full}, we will produce an integral representation for a multisequence whose infinite $d$-Hankel matrix has finite rank. In Section \ref{sec:RPE}, we will prove a number of key results on extensions of $d$-Hankel matrices for future use. In Section \ref{sec:CFana}, we will formulate and prove both indefinite analogues of Curto and Fialkow's flat extension theorem described above. In Section \ref{sec:Ex}, we will compute a $12$ atomic quasicomplex and a $15$ atomic signed representing measure for a bivariate degree $6$ sequence such that $M(3)$ is positive semidefinite yet the sequence does not have a positive representing measure which appeared in \cite{Fialkow_FV}. Finally, in Appendix \ref{app:P}, we will provide some requisite background on Pontryagin spaces.

For the convenience of the reader we have compiled a list of commonly used notation that appears throughout the paper.

\smallskip

\noindent {\bf Notation}

\smallskip

\noindent $\RR^{\ell \times \ell}$ and $\CC^{\ell \times \ell}$ denote the vector spaces of real and complex matrices of size $\ell \times \ell$, respectively

\smallskip

\noindent $\RR^{\ell} := \RR^{\ell \times 1}$ and $\CC^{\ell} :=\CC^{\ell \times 1}$

\smallskip

\noindent {\bf 0} denotes the zero vector in $\CC^{\ell}$.

\smallskip

\noindent $A^*$ and $A^T$ denote the conjugate transpose and transpose, respectively, of a matrix $A \in \CC^{\ell \times \ell}$

\smallskip

\noindent $\cC_A$ and $\cN_A$ denote the column space and null space, respectively, of a matrix $A \in \CC^{\ell \times \ell}$

\noindent $i_{\pm}(A)$ denote the number of positive and negative eigenvalues (counting algebraic multiplicity), respectively, of a Hermitian matrix $A$

\smallskip

\noindent $\sigma(A)$ denotes the spectrum of a matrix $A \in \CC^{\ell \times \ell}$

\smallskip

\noindent $A \succeq 0$ if $A$ is positive semidefinite

\smallskip

\noindent $\col(c_{\lambda})_{\lambda \in \Lambda}$ and $\row(c_{\lambda})_{\lambda \in \Lambda}$ denote the column and row vectors with entries $(c_{\lambda})_{\lambda \in \Lambda}$, respectively.

\noindent $\NN_0^d$, $\RR^d$ and $\CC^d$ denote the set of $d$tuples of nonnegative integers, real numbers and complex numbers, respectively.

\smallskip

\noindent $|\gamma| := \displaystyle\sum_{j=1}^d \gamma_j$ for $\gamma = (\gamma_1, \ldots, \gamma_d) \in \NN_0^d$

\smallskip

\noindent $z^{\gamma} := \displaystyle\prod_{j=1}^d x_j^{\gamma_j}$ for $z = (z_1, \ldots, z_d) \in \CC^d$ and $\gamma = (\gamma_1, \ldots, \gamma_d) \in \NN_0^d$

\smallskip

\noindent ${\rm e}_j := \in \NN_0^d$ denotes the $d$tuple of $0$s with a $1$ in the $j$ position

\smallskip

\noindent $0_d := (0, \ldots, 0) \in \NN_0^d$

\smallskip

\noindent $\pi_j(\gamma) := \gamma_j$ for for $\gamma = (\gamma_1, \ldots, \gamma_d) \in \NN_0^d$.

\smallskip

\noindent $\preceq_L$ denotes the graded lexicographic ordering on $\NN_0^d$

\smallskip

\noindent $\ell(m, d) := \binom{m+d}{d} = \frac{(m+d)! }{n! \, d!}$

\smallskip

\noindent $\RR[x_1, \ldots, x_d]$ denotes the ring of polynomials in $d$ indeterminate variables with real coefficients

\smallskip

\noindent $\RR^d_n[{\bf x}] := \{ p \in \RR[x_1, \ldots, x_d]: p(x) = \sum_{0 \leq |\gamma| \leq n} p_{\gamma} x^{\gamma} \}$

\noindent $\mu = \mu_+ - \mu_-$ will denote the Jordan decomposition (see, e.g., \cite[p. 119]{Rudin}) of a signed Borel measure $\mu$.

\section{$d$-Hankel matrices} 
\label{sec:MHM}
In this section we will define $d$-Hankel matrices (also known as moment matrices in the literature) and introduce a number of basic facts for future use.

\begin{defn}[$d$-Hankel matrix]
\label{def:Feb13dH}
Given a real-valued truncated multisequence \\$s = (s_{\gamma})_{\stackrel{0 \leq |\gamma| \leq 2n}{\gamma \in \NN_0^d}}$, let
\begin{equation}
\label{eq:ell3}
\ell(n,d) = \binom{n+d}{d} = \frac{(n+d)!}{n!\, d!}
\end{equation}
and $M(n) \in \RR^{\ell(n,d) \times \ell(n,d)}$ denote the {\it d-Hankel} matrix based on $s$ constructed in the following way. First order the monomials 
$$(X^{\lambda})_{0 \leq |\lambda| \leq n} = (X_1^{\lambda_1}\cdots X_d^{\lambda_d})_{\stackrel{\lambda = (\lambda_1, \ldots, \lambda_d)}{0 \leq |\lambda| \leq n} }$$
in a graded lexicographic manner, i.e., use the ordering
$$X^{\lambda} \preceq_L X^{\gamma} \quad {\rm with} \quad \lambda = (\lambda_1, \ldots, \lambda_d) \quad {\rm and} \quad \gamma = (\gamma_1, \ldots, \gamma_d)$$
if $|\lambda| < |\gamma|$, and if $|\lambda| = |\gamma|$, then
$\lambda_j \geq \gamma_j$ for $j=1,\ldots, d$ and $\lambda_d \leq \gamma_d$. 
For example, if, $d=2$ and $n=3$, then
$$1, X_1, X_2, X_1^2, X_1X_2, X_2^2, X_1^3, X_1^2X_2, X_1X_2^2, X_2^3.$$
Note that 
$$\card(X^{\lambda})_{0 \leq |\lambda| \leq n} = \ell(n,d) = \binom{n+d}{d}.$$
Thus, we may index the rows and columns of $M(n)$ by $(X^{\lambda})_{0 \leq |\lambda| \leq n}$ (ordered by $\preceq_L$) and let the entry in the row indexed by $X^{\lambda}$ and the column indexed by $X^{\xi}$ be given by
$$s_{\lambda+\xi}.$$
Following \cite{CF_memoir} (and subsequent papers of Curto and Fialkow), we shall use $X^{\lambda} \in \cC_{M(n)}$, where $\cC_{M(n)}$ denotes the column space of $M(n)$, to refer to the column vector in $\RR^{\ell(n,d)}$. 
\end{defn}


\begin{example}
\label{ex:MV1}
If $d = 2$ and $s = (s_{\gamma})_{0 \leq |\gamma|\leq 4}$, then
$$M(2) = \kbordermatrix{ & 1 & X_1 & X_2 & X_1^2 & X_1X_2 & X_2^2 \\
1 & s_{00} & s_{10} & s_{01} & s_{20} & s_{11} & s_{02} \\ 
X_1 & s_{10} & s_{20} & s_{11} & s_{30} & s_{21} & s_{12} \\
X_2 & s_{01} & s_{11} & s_{02} & s_{21} & s_{12}  & s_{03} \\
X_1^2 & s_{20} & s_{30} & s_{21} & s_{40} & s_{31} & s_{22} \\
X_1 X_2 & s_{11} & s_{21} & s_{12} & s_{31} & s_{22} & s_{13} \\
X_2^2 & s_{02} & s_{12} & s_{03} & s_{22} & s_{13} & s_{04} 
}.$$
\end{example}

Fix $d \in \NN$ and $n \in \NN_0$. We may index the rows and columns of any given Hermitian matrix $M \in \RR^{\ell(n,d) \times \ell(n,d)}$, where $\ell(n,d)$ is as in \eqref{eq:Feb10po3} by the monomials $(X^{\gamma})_{0 \leq |\gamma| \leq n}$ (ordered in a graded lexicographic manner). Thus, we may write
$$M = (m_{\lambda, \xi})_{0 \leq |\lambda|, |\xi| \leq n}$$
and introduce the Hermitian form on $\RR^d_n[{\bf x}]$, i.e., the space of polynomials of total degree at most $n$ with complex coefficients, given by 
\begin{equation}
\label{eq:Feb10po3}
(x^{\lambda}, x^{\xi})_M = m_{\lambda, \xi}.
\end{equation}

In the following lemma, $\pi_j(\lambda)$ will denote the canonical mapping of $\NN_0^d$ onto $\NN_0$ given by
$$\pi_j(\lambda) = \lambda_j \quad\quad {\rm for} \quad \lambda = (\lambda_1, \ldots, \lambda_d) \in \NN_0^d.$$

\begin{lem}
\label{lem:dH}
Fix $d \in \NN$ and $n \in \NN_0$. Let $M \in \RR^{\ell(n,d) \times \ell(n,d)}$ and write $M= (m_{\lambda,\xi})_{\stackrel{0 \leq |\lambda|, |\xi| \leq n}{\lambda, \xi \in \NN_0^d}}$ and $(\cdot, \cdot)_M$ be as above. $M$ is $d$-Hankel if and only if 
\begin{equation}
\label{eq:dhan3}
(x^{\lambda}, x^{\xi} )_M = (x^{\gamma}, x^{\eta})_M \quad\quad {\rm for} \quad 0 \leq |\lambda|, |\xi|, |\gamma|, |\eta| \leq n
\end{equation}
whenever 
$$
\pi_j(\lambda) + \pi_j(\xi) = \pi_j(\gamma) + \pi_j(\eta)
$$
for $0 \leq |\lambda|, |\xi|, |\gamma|,|\eta| \leq n$ and $j=1,\ldots, d$.
\end{lem}

\begin{proof}
Condition \eqref{eq:dhan3} is equivalent to
\begin{equation}
\label{eq:dhan4}
m_{\lambda, \xi} = m_{\gamma, \eta}
\end{equation}
whenever
$$
\pi_j(\lambda) + \pi_j(\xi) = \pi_j(\gamma) + \pi_j(\eta)
$$

for $0 \leq |\lambda|, |\xi|, |\gamma|,|\eta| \leq n$ and $j=1,\ldots, d$. Thus, if we let $s = (s_{\gamma})_{0 \leq |\gamma| \leq 2n}$ be given by
$$s_{\lambda+ \xi } = m_{\lambda, \xi},$$
then \eqref{eq:dhan4} implies that
$$M = M(n).$$
Thus, $M$ is $d$-Hankel.
\end{proof}

\begin{lem}
\label{lem:Mar5tg1}
Suppose that $\lambda, \xi, \gamma$ and $\eta$ belong to $\NN_0^d$ and satisfy
$$0 \leq |\lambda|, |\xi|, |\gamma|, |\eta| \leq n.$$
If
\begin{equation}
\label{eq:Mar5pp3}
\lambda_j + \xi_j = \gamma_j + \eta_j \quad {\rm for} \quad j =1,\ldots, d
\end{equation}
and
$$\xi \preceq_L \eta,$$
then
\begin{equation}
\label{eq:Mar5k3}
|\lambda|+|\xi | = |\gamma|+ | \eta|
\end{equation}
and
$$\gamma \preceq_L \lambda$$
\end{lem}

\begin{proof}
Assertion \eqref{eq:Mar5k3} is an easy consequence of \eqref{eq:Mar5pp3}. Suppose to the contrary that $\lambda \preceq_L \gamma$ and $\lambda \neq \gamma$. In view of \eqref{eq:Mar5k3}, the only case we need consider is when
$$\text{$|\xi| = |\eta|$ with $\xi_j \geq \eta_j$ for $j=1,\ldots, d-1$ and $\xi_d \leq \eta_d$}$$ 
and
$$\text{$|\lambda|= |\gamma|$ with $\lambda_j \geq \gamma_j$ for $j=1,\ldots, d-1$ and $\lambda_d \leq \gamma_d.$}$$
Since $\lambda \neq \gamma$, one of the above inequalities must be a strict inequality. Thus, assumption \eqref{eq:Mar5pp3} cannot possibly hold.
\end{proof}

\begin{lem}
\label{lem:dHHerm}
Let $n > 0$ and $s = (s_{\gamma})_{\stackrel{0 \leq |\gamma| \leq 2n}{\gamma \in \NN_0^d} }$ be a given real-valued truncated multisequence. Any submatrix $\Phi$ of $M(n)$ with rows indexed by 
$$(X^{\lambda})_{\lambda \in \Lambda}, \quad \quad {\rm where} \quad \emptyset \neq \Lambda \subseteq \{ \gamma \in \NN_0^d: 0 \leq |\gamma| \leq n \},$$
and columns indexed by
$$(X^{\lambda+ \xi})_{\lambda \in \Lambda} \quad \text{for any fixed $\xi \in \NN_0^d$ with $|\lambda +\xi| \leq n$}$$
is Hermitian.
\end{lem}

\begin{proof}
Write $\Phi = (\varphi_{\lambda, \gamma + \xi})_{\stackrel{0 \leq |\lambda| \leq n-1}{0 \leq |\gamma | \leq n-1}}$, where $\varphi_{\lambda, \gamma + \xi} = s_{\lambda+\gamma+\xi}$. Then we must check that
$$\varphi_{\lambda, \gamma + \xi} = \varphi_{\gamma, \lambda + \xi}.$$
But this is immediate since
$$\varphi_{\gamma,\lambda+ \xi} = s_{\gamma+\lambda+\xi}.$$
\end{proof}

The following lemma serves as a guide for constructing rank preserving extensions of $d$-Hankel matrices. In the case that $A \succeq 0$, then the following lemma is well-known and is due to Smuljan \cite{Smuljan}. 

\begin{lem}
\label{lem:Mar31e3}
Let $A = A^* \in \CC^{m \times m}$, $B \in \CC^{m \times n}$ and $C = C^* \in \CC^{n \times n}$ be given. The block matrix
$$M := \begin{pmatrix}
A & B \\ B^* & C 
\end{pmatrix}$$
satisfies
\begin{equation}
\label{eq:r1}
\rank M = \rank A
\end{equation}
if and only if there exists a matrix $W \in \CC^{m \times n}$ such that
\begin{equation}
\label{eq:Mar31r1}
B = AW
\end{equation}
and
\begin{equation}
\label{eq:Mar31r2}
C = W^* A W.
\end{equation}
In this case, $i_{\pm}(A) = i_{\pm}(M)$.
\end{lem}

\begin{proof}
If \eqref{eq:Mar31r1} and \eqref{eq:Mar31r2} are in force, then
\begin{equation}
\label{eq:Mar31f1}
M := 
\begin{pmatrix}
A & B \\ B^* & C 
\end{pmatrix}
= 
\begin{pmatrix}
I & 0 \\ W^* & I
\end{pmatrix}
\begin{pmatrix}
{A} & 0 \\ 0 & 0
\end{pmatrix}
\begin{pmatrix}
I & W \\ 0 & I
\end{pmatrix}.
\end{equation}
Thus, Sylvester's law of inertia (see, e.g., Theorem 4.5.8 in \cite{HJ}) implies that
$i_{\pm}(A) = i_{\pm}(M)$ and hence 
$$
\rank A = i_{+}(A) + i_-(A) = i_+(M) + i_-(M) = \rank M.
$$

Conversely, suppose \eqref{eq:r1} is in force. Let $\cC_A$ and $\cN_A$ denote the column space and null space, respectively, of $A$. Write
$$A = \begin{pmatrix}
\wt{A} & 0 \\ 0 & 0
\end{pmatrix}:
\begin{array}{cccc}
\cC_A & & \cC_A \\
\oplus & \to & \oplus \\
\cN_A & & \cN_A
\end{array}$$
and
$$B^* = \begin{pmatrix}
B_{11}^* & B_{21}^*
\end{pmatrix}:
\begin{array}{cccc}
\cC_A  & &  \\
\oplus & \to & \CC^n. \\
\cN_A  & & 
\end{array}$$
We may use a Schur complement argument and Sylvester's law of inertia to realise
\begin{align*}
i_{\pm}(M) =& \; i_{\pm} \begin{pmatrix} \wt{A} & 0 & B_{11} \\ 0 & 0 & B_{21} \\ B_{11}^* & B_{21}^* & C \end{pmatrix} \\
=& \; i_{\pm} \begin{pmatrix} 
\wt{A} & 0 & 0 \\ 0 & 0 & B_{21} \\ 0 & B_{21}^* & c - B_{11}^* \wt{A}^{-1} B_{11}
\end{pmatrix} \\
=& \; i_{\pm}(\wt{A}) + i_{\pm} \begin{pmatrix}
0 & B_{21} \\ B_{21}^* & C - B_{11}^* \wt{A}^{-1} B_{21}
\end{pmatrix}
\end{align*} 
and hence 
$$\rank \begin{pmatrix}
0 & B_{21} \\ B_{21}^* & C - B_{11}^* \wt{A}^{-1} B_{11}
\end{pmatrix} = 0,$$
i.e., 
$B_{21} = 0$ and $C = B_{11}^* \wt{A}^{-1} B_{11}$. Finally, if we let $W = \wt{A}^{-1} B_{11}$, then it can easily be checked that \eqref{eq:Mar31r1} and \eqref{eq:Mar31r2} hold.
\end{proof}


%

\section{Quasi-complex measures}
\label{sec:QCM}
In this section, we will introduce a class of representing measures that will be considered in the sequel.

\begin{defn}[Quasi-complex measure on $\CC^d$]
\label{def:Feb22qcm}
A Borel measure $\mu$ on $\CC^d$ is called {\it quasi-complex} if the support of $\mu$, denoted by $\supp \mu$ is symmetric about the real axis, i.e., 
$$z \in \supp \mu \Longleftrightarrow \bar{z} \in \supp \mu.$$
\end{defn}

\begin{rem}
\label{rem:Feb23possigned}
It is immediate that every signed Borel measure is quasi-complex and, in particular, every positive Borel measure is also quasi-complex.
\end{rem}

\begin{defn}[Quasi-complex and signed representing measure]
Given a real-valued multisequence $s = (s_{\gamma})_{\stackrel{0 \leq |\gamma| \leq m}{\gamma \in \NN_0^d}}$, where $0 \leq m \leq \infty$, and a quasi-complex measure $\mu$ on $\CC^d$ such that
$$\int_{\CC^d} |z^{\gamma}| d\mu(z) := \int \cdots \int_{\CC^d} \left| z_1^{\gamma_1} \cdots z_d^{\gamma_d} \right| d\mu(z_1, \ldots, z_d) < \infty$$
for all $0 \leq |\gamma| \leq m$, we say that $\mu$ is a quasi-complex representing measure for $s$ if
\begin{equation} 
\label{eq:Feb22rm}
s_{\gamma} = \int_{\CC^d} z^{\gamma} d\mu(z) \quad {\rm for} \quad 0 \leq |\gamma| \leq m.
\end{equation}
Finally, if, in particular, $\mu$ is a quasicomplex representing measure for $s$ such that $\mu$ is real-valued and $\supp \mu \subseteq \RR^d$, then we will call $\mu$ a signed representing measure for $s$.
\end{defn}

\begin{defn}[An analogue of a Vandermonde matrix]
\label{def:VM}
Corresponding to 
$$z^{(1)}, \ldots, z^{(r)} \in \CC^d \quad {\rm and} \quad \quad n \in \NN_0,$$ we may define the Vandermonde matrix
$V \in \CC^{r \times \ell(n,d)}$, where $\ell(n,d) = \binom{n+d}{d}$, in the following manner. Index the rows of $V$ by $1, \ldots, r$ and the columns of $V$ by $\{ \lambda \in \NN_0^d: 0 \leq |\lambda | \leq n \}$ (ordered by $\preceq_L$ as in Section \ref{sec:MHM}) and let the entry in the row indexed by $q$ and the column indexed by $\lambda$ be given by
$$(z^{(q)} )^{\lambda}.$$
\end{defn}

\begin{lem}
\label{lem:Feb22m1}
Let $s = (s_{\gamma})_{0 \leq |\gamma| \leq 2n}$ be a given real-valued truncated multisequence. If $s$ has an $r$-atomic quasi-complex representing measure $\mu = \sum_{q=1}^r \varrho_a \delta_{z^{(a)}}$, then
\begin{equation}
\label{eq:Feb22f1}
M(n) = V^T R V,
\end{equation}
where $M(n)$ is as in Definition {\rm \ref{def:Feb13dH}}, $V$ is as above and $R = {\rm diag}(\varrho_1, \ldots, \varrho_r)$.
\end{lem}

\begin{proof}
If $\mu$ is an $r$-atomic quasi-complex representing measure for $s$, then
$$s_{\gamma} = \sum_{a=1}^r \varrho_q (z^{(a)})^{\gamma} \quad {\rm for } \quad 0 \leq |\gamma| \leq 2n$$
and hence the validity of the factorisation for $M(n)$ claimed in \eqref{eq:Feb22f1} is easily checked.
\end{proof}

\begin{lem}
\label{lem:Feb22supp}
Let $s = (s_{\gamma})_{0 \leq |\gamma| \leq m}$ be a given real-valued truncated multisequence. Then any quasi-complex representing measure $\mu$ for $s$ must obey
\begin{equation}
\label{eq:Feb22pp4}
\rank M(\lfloor m/2 \rfloor) \leq \card \supp \mu. 
\end{equation}
\end{lem}

\begin{proof}
If $\supp \mu$ is not finite, then \eqref{eq:Feb22pp4} holds automatically. If $\card \supp \mu = r < \infty$, then we may write $\mu = \sum_{a=1}^r \varrho_a \delta_{z^{(a)}}$ and use the factorisation \eqref{eq:Feb22f1} to obtain the estimate
$$\rank M(n) \leq \rank R = r = \card \supp \mu.$$
\end{proof}

\section{A Pontryagin space defined by $( s_{\gamma} )_{\gamma \in \NN_0^d}$}
\label{sec:Pontryagin}

Let $s = ( s_{\gamma} )_{\gamma \in \NN_0^d}$ be a real-valued multisequence with $\kappa$-negative squares (see Definition \ref{def:Oct27kns}). In this section, we will introduce a Pontryagin space based on $s$. In the case that $\kappa = 0$, then this construction collapses to the usual Hilbert space construction. For $\lambda \in \NN_0^d$, consider the translation operator
$$(E_{\lambda} s)(\gamma) = s_{\lambda + \gamma} \quad\quad  {\rm for} \quad \gamma \in \NN_0^d.$$
In view of $E_{\lambda} E_{{\xi}} = E_{\lambda + \xi}$, the real linear span of translation operators $E_{\lambda}, \lambda \in \NN_0^d$ forms an algebra which we will denote by $\sA^d$. It is readily checked that 
$$\{ E_{\lambda}: \lambda \in \NN_0^d \}$$
are linearly independent in $\sA^d$. Consequently, every element $p(E) \in \sA^d$ admits the representation
$$p(E) = \sum_{0 \leq |\lambda| \leq m} c_{\lambda} E_{\lambda} \quad \quad {\rm for }\quad E_{\lambda} \in \sA^d.$$

In what follows, the number of negative eigenvalues (counting algebraic multiplicity) of a matrix $M \in \CC^{\ell \times \ell}$ will be denoted by $i_-(M)$.

\begin{defn}
\label{def:Oct27kns}
Let $s = (s_{\gamma})_{\gamma \in \NN_0^d}$ be a given real-valued sequence. We say that $s$ has $\kappa$-negative squares if there exists a nonnegative integer $\kappa$ such that the $d$-Hankel matrices satisfy
$$i_-(M(n)) \leq \kappa \quad {\rm for} \quad n =0,1,\ldots$$
and
$$i_-(M(n_0)) = \kappa \quad \text{for some $n_0 \in \NN_0$}.$$
If $\kappa = 0$, then $i_-(M(n)) = 0$ for $n=0,1,\ldots$ and we arrive at the usual notion of a {\it positive definite sequence}.
\end{defn}

Corresponding to a real-valued multisequence $s = ( s_{\gamma} )_{\gamma \in \NN_0^d}$ with $\kappa$-negative squares, consider the space
\begin{equation}
\label{eq:Oct27yj3}
T(s) := \{ p(E) s : p(E) \in \sA^d \}
\end{equation}
endowed with the Hermitian form
$$
(p(E), q(E))_s = \sum_{0 \leq |\lambda| \leq m} \sum_{0 \leq |\xi| \leq n}  c_{\lambda} {d}_{\xi} s_{\lambda + \xi},
$$
where 
$$p(E) = \sum_{0 \leq |\lambda| \leq m} c_{\lambda} E_{\lambda} \quad {\rm and} \quad
q(E) =  \sum_{0 \leq |\xi| \leq n} d_{\xi} E_{\xi}.$$
It is readily checked that $(p(E), q(E))_s$ does not depend on the particular representations for $p(E)$ and $q(E)$ appearing above. Moreover, $(p(E), q(E))_s = {(p(E),q(E))_s}$. If $p(E)s = \sum_{0 \leq |\lambda| \leq m} c_{\lambda} E_{\lambda}s \in T(s)$, then
\begin{align}
(p(E), E_{\xi} )_s =& \; \left( \sum_{0 \leq |\lambda| \leq m} c_{\lambda} E_{\lambda}, E_{\xi}  \right)_s \nonumber \\
=& \; \sum_{0 \leq |\lambda| \leq m} c_{\lambda}s_{\lambda+\xi} \nonumber \\
=& \; (p(E)s)(\xi) \quad {\rm for} \quad p(E)s \in T(s) \quad {\rm and} \quad \xi \in \NN_0^d \label{eq:Oct27rp3},
\end{align}
and hence, $(p(E), q(E))_s = 0$ for all $q(E)s \in T(s)$ if and only if $p(E)s = 0$, i.e.,
\begin{equation}
\label{eq:Oct28pp3}
\text{$(\cdot, \cdot)_s$ is not degenerate.}
\end{equation}
Another consequence of \eqref{eq:Oct27rp3} is
$$(p(E),q(E) )_s = (q(E) p(E)s)(0_d).$$

\begin{rem}
\label{rem:Oct27tn3}
For any real polynomial $p(x) = \sum_{0 \leq |\gamma| \leq n} p_{\gamma} x^{\gamma}$ belonging to $\RR[x_1, \ldots, x_d]$, we can form the operator
$p(E)$ on $\sA^d$. It is easy to see that $p \mapsto p(E)$ is an isomorphism of $\RR[x_1, \ldots, x_d]$ onto $\sA^d$.
\end{rem}

The following result appeared in \cite{BergSasvari} in a more general semigroup setting. We will provide a proof adapted from \cite{BergSasvari} to our current context for completeness.

\begin{thm}[Proposition 3.4 in \cite{BergSasvari} with $S = \NN_0^d$]
\label{thm:Oct27yj3}
Let $s = (s_{\gamma})_{\gamma \in \NN_0^d}$ be a given real-valued sequence with $\kappa$-negative squares. Then there exists a real Pontryagin space $\Pi_{\kappa}(s)$ with the following properties{\rm :}
\begin{enumerate}
\item[(i)] $T(s)$ is dense in $\Pi_{\kappa}(s)$.
\item[(ii)] $t_{\gamma} = ( t, E_{\gamma} )_s$ for $\gamma \in \NN_0^d$ and $t \in \Pi_{\kappa}(s)$.
\end{enumerate}
\end{thm}

\begin{proof}
Let $T(s)$ be the pre-$\pi_{\kappa}$ space constructed in \eqref{eq:Oct27yj3}. Now consider the orthogonal decomposition
\begin{equation}
\label{eq:Oct27od3}
T(s) = \cP \oplus \cN,
\end{equation}
where $\cP$ is a pre Hilbert space and $\cN$ is a negative subspace of dimension $\kappa$. If $p(E) \in \cP$, then \eqref{eq:Oct28pp3} and \eqref{eq:Oct27uu3} ensure
$$|(p(E)s)(\gamma)| = | (p(E), E_{\gamma})_s | \leq \sqrt{(p(E), p(E))_s} \| E_{\gamma} \|_s.$$
Thus, if $(u^{(n)}(E)s)_{n=1}^{\infty}$ is a Cauchy sequence in $\cP$, then $((u^{(n)}(E)s)(\gamma))_{n=1}^{\infty}$ is a Cauchy seuqence in $\RR$ for any fixed $\gamma \in \NN_0^d$. Since $\RR$, with the usual metric, is complete we have the existence of a complex number $(u(E)s)(\gamma)$ such that
$$\lim_{n \uparrow \infty} (u^{(n)}(E)s)(\gamma) = (u(E)s)(\gamma) \quad {\rm for} \quad \gamma \in \NN_0^d.$$ 
Let $\Pi^+_{\kappa}(s)$ be the Hilbert space completion and put
$$\Pi_{\kappa}(s) = \Pi_{\kappa}^+(s) \oplus \Pi_{\kappa}^-(s),$$
where $\Pi_{\kappa}^-(s) := \cN$. Thus, $\Pi_{\kappa}(s)$ is a Pontryagin space with $\kappa$-negative squares.

It is worth noting that although the decomposition \eqref{eq:Oct27od3} is far from unique, in general, all constructions of $\Pi_{\kappa}(s)$ in this manner are isomorphic. Finally, as the mapping $t \mapsto (t, E_{\gamma} s)_s$ is continuous on $\Pi_k(s)$ and $T(s)$ is dense in $\Pi_{\kappa}(s)$ the identity \eqref{eq:Oct27rp3} extends to any $t \in \Pi_{\kappa}(s)$. 
\end{proof}

\section{An integral representation for $(s_{\gamma})_{\gamma \in \NN_0^d}$ with $\rank M(\infty) < \infty$}
\label{sec:full}

Given a real-valued $d$-multisequence $s = (s_{\gamma})_{\gamma \in \NN_0^d}$, let
$$(X^{\gamma})_{\gamma \in \NN_0^d}:= (X_1^{\gamma_1} \cdots X_d^{\gamma_d})_{\gamma = (\gamma_1 \cdots \gamma_d) \in \NN_0^d}$$ 
denote all monomials in the real indeterminates $(X_1, \ldots, X_d)$ ordered in a graded lexicographic manner, i.e., 
$$1, X_1, \ldots, X_d, X_1^2, X_1X_2 \ldots, X_{d-1}X_d, X_d^2, \ldots.$$ 
Given $s = (s_{\gamma})_{\gamma \in \NN_0^d}$, we shall let $M(\infty)$ denote the linear operator on
$$\RR_0^{\omega} = \{ v = (v_{\gamma})_{\gamma \in \NN_0^d}: \text{$v_{\gamma} \in \RR$ and is $0$ for all but finitely many $\gamma \in \NN_0^d$}\}.$$

Following \cite{CF_memoir}, we will denote an element of the column space $\cC_{M(\infty)}$ by 
$$p(X) \quad\quad {\rm for} \quad p(x) = \sum_{0 \leq |\gamma| \leq m} p_{\gamma} x^{\gamma} \in \RR[x_1, \ldots, x_d]$$
and
\begin{equation}
\label{eq:Feb17rank1}
\rank M(\infty) := \dim \cC_{M(\infty)}.
\end{equation}

\begin{lem}
\label{lem:Feb17rl3}
Let $s = (s_{\gamma})_{\gamma \in \NN_0^d}$ be a given real-valued multisequence. Then
\begin{equation}
\label{eq:Feb17ranksup3}
\rank M(\infty) = \sup_{n \in \NN} M(n).
\end{equation}
If $\rank M(\infty) < \infty$, then $\sup$ can be replaced by $\max$ in \eqref{eq:Feb17ranksup3}.
\end{lem}

\begin{proof}
Suppose $\rank M(\infty)$ is not finite and let $\Gamma = \{ X^{\gamma^{(m)}} \}_{m=1}^{\infty}$ be a basis for $\cC_{M(\infty)}$. Let $\Gamma_n$ be a nonempty subset of $\Gamma$ with $\card \Gamma_n = n$ and $\Phi(\Gamma_n) \in \RR^{n \times n}$ denote the principal submatrix of $M(\infty)$ with rows and columns indexed by $\Gamma_n$. The matrix $\Phi(\Gamma_n)$ can be realised as a submatrix of $$M(m_n) \quad \quad {\rm for} \quad m_n := \max_{\gamma \in \Gamma_n} |\gamma|\in \NN_0.$$ Consequently,
$$\rank M(m_n) \geq n$$
and thus we have \eqref{eq:Feb17ranksup3} when $\rank M(\infty)$ is not finite.

If $r := \rank M(\infty) < \infty$, then let $(X^{\gamma})_{\gamma \in \Gamma_r}$ be a basis for $\cC_{M(\infty)}$. Then let $\Phi(\Gamma_r) \in \RR^{r \times r}$ be the principal submatrix of $M(\infty)$ with rows and columns indexed by $\Gamma_r$ and
$$n_r := \max_{\gamma \in \Gamma_r} |\gamma|.$$
Then for $n \geq n_r \in \NN_0$, 
$\Phi(\Gamma_n)$ is a submatrix of $M(n)$. Thus,
$$\rank M(n) \geq r.$$
On the other hand, since $(X^{\lambda})_{\lambda \in \NN_0^d \, \setminus \, \Gamma_n}$ can be written as
$$X^{\lambda} = \sum_{\gamma \in \Gamma_n} c_{\gamma} X^{\gamma},$$
we have 
$$\rank M(n) \leq r \quad\quad {\rm for} \quad n \geq n_r.$$ 
Thus, $\rank M(n) = r$ for $n \geq n_r$ and we have \eqref{eq:Feb17ranksup3} with $\sup$ replaced by $\max$.
\end{proof}

\begin{cor}
\label{cor:Feb17rank}
Let $s = (s_{\gamma})_{\gamma \in \NN_0^d}$ be a given real-valued multisequence. If
$$r := \rank M(\infty) < \infty,$$
then $s$ has $\kappa$-negative squares for some $\kappa \leq r$.
\end{cor}

\begin{proof}
If $r := \rank M(\infty) < \infty$, then it follows from Lemma \ref{lem:Feb17rl3} that
$$\max_{n \in \NN_0} \rank M(n) = r.$$
Since $M(n)$ is Hermitian,
$$i_-(M(n)) + i_+(M(n)) = \rank M(n) = r,$$
where $i_-(M(n))$ and $i_+(M(n))$ denote the number of negative and the number of positive eigenvalues of the Hermitian matrix $M(n)$. Thus, 
$$i_-(M(n)) \leq \rank M(n) \quad \quad {\rm for} \quad n \in \NN_0$$
and consequently $s$ has a $\kappa$-negative squares for some $\kappa \leq r$.
\end{proof}

\begin{lem}
\label{lem:ideal}
Let $s = (s_{\gamma})_{\gamma \in \NN_0^d}$ be a given real-valued multisequence. If
\begin{equation}
\label{eq:i1}
\cI = \{ p \in \RR[x_1, \ldots, x_d]: p(X) = {\bf 0} \},
\end{equation}
then $\cI$ is an ideal of $\RR[x_1, \ldots, x_d].$
\end{lem}

\begin{proof}
If $\cI = \emptyset$, then there is nothing to prove. If $\cI \neq \emptyset$, then the only nontrivial thing to check is that if $p(x) = \sum_{0 \leq |\gamma| \leq m} p_{\gamma} x^{\gamma}$ and $p(X) = {\bf 0}$, then 
\begin{equation}
\label{eq:c1}
(pq)(X) = {\bf 0} \quad \quad {\rm for} \quad q \in \RR[x_1, \ldots, x_d].
\end{equation}
If $p(X) = {\bf 0}$, then
\begin{equation}
\label{eq:c2}
\sum_{0 \leq |\gamma| \leq m} p_{\gamma} {\rm col}(s_{\gamma+\lambda})_{\lambda \in \NN_0^d} = {\bf 0}.
\end{equation}
But as
$$(x^{\xi}p)(X) = {\rm col}(s_{\gamma+\xi+\eta})_{\eta \in \NN_0^d} = {\bf 0}$$
is verified by putting $\lambda = \xi + \eta$ in \eqref{eq:c2}, 
\eqref{eq:c1} follows by linearity. 
\end{proof}

Given $s= (s_{\gamma})_{\gamma \in \NN_0^d}$, consider the quotient space
\begin{equation}
\label{eq:Feb15q1}
\cA = \RR[x_1, \ldots, x_d] \, / \, \cI,
\end{equation}
endowed with the Hermitian form 
\begin{equation}
\label{eq:June20a3}
( p+\cI, q + \cI )_{\cA} = (p(E), q(E))_s \quad \quad {\rm for} \quad p,q, \in \RR[x_1, \ldots, x_d].
\end{equation}
We will be in need of the shift operators $\Theta_j: \cA \to \cA$ for $j=1,\ldots,d$, given by
\begin{equation}
\label{eq:Feb15shift}
\Theta_j(p+\cI) = x^{ {\rm e}_j} p(x) + \cI \quad \text{for $p+\cI \in \cA$ and $j=1,\ldots, d$.}
\end{equation}

\begin{lem}
\label{lem:Oct28isomorph0}
Let $s = (s_{\gamma})_{\gamma \in \NN_0^d}$ be a given real-valued multisequence. If $r = \rank M(\infty) < \infty$, then $T(s)$, $\cA$ and $\cC_{M(\infty)}$ {\rm (}endowed with the Hermitian form $(\cdot, \cdot)_s${\rm )} are isomorphic Pontryagin spaces with $\kappa$-negative squares.
\end{lem}

\begin{proof}
The fact that $T(s)$ and $\cC_{\infty}$ are isomorphic is immediate (see Definition \ref{def:June20d1} for a precise definition of what it means for two Pontryagin spaces to be isomorphic and also Remark \ref{rem:June20r1} if need be). 
We will now show that $\cA$ and $\cC_{M(\infty)}$ are isomorphic. Let $\varphi: \cA \to \cC_{M(\infty)}$ be given by
$$\varphi(p + \cI) = p(X) \quad \quad {\rm for} \quad p \in \RR[x_1, \ldots, x_d].$$
$\varphi$ is obviously linear, well-defined and surjective. To see that $\varphi$ is injective, note that
$$\varphi(p + \cI) = \varphi(q+ \cI) \Longleftrightarrow p(X) = q(X) \Longleftrightarrow (p-q)(X) = {\bf 0} \Longleftrightarrow p - q \in \cI.$$
Finally, as
$$(\varphi(p+\cI), \varphi(q+\cI) )_{\cA} = (p, q)_s,$$
we have the desired isomorphism. 
\end{proof}

\begin{lem}
\label{lem:Oct28isomorph10}
Suppose $s = (s_{\gamma})_{\gamma \in \NN_0^d}$ is a given real-valued multisequence and 
$$r:= {\rm rank} \, M(\infty) < \infty.$$ Then the variety of the ideal $\cI \subseteq \RR[x_1, \ldots, x_d]$, i.e.,
$$\cV(\cI) := \{ z \in \CC^d: \text{$p(z)  =0$ {\rm for all} $p \in \cI$}\}$$
satisfies
\begin{equation}
\label{eq:Feb15d4}
\ell := \card \cV(\cI) \leq r.
\end{equation}
\end{lem}

\begin{proof}
See, e.g., Theorem 2.6 in \cite{Laurent}.
\end{proof}

Let $s$, $r$ and $\ell$ be as in Lemma \ref{lem:Oct28isomorph0}. Fix a basis $\cB := \{ p^{(1)}+ \cI, \ldots, p^{(r)} + \cI \}$ of $\cA = \RR[x_1, \ldots, x_d] / \cI$ (it follows immediately from Lemma \ref{lem:Oct28isomorph0} and Remark \ref{rem:June20r1} and write
that $\dim \cA = r$) and write
$$\cV(\cI) = \{ z^{(1)}, \ldots, z^{(\ell)} \}.$$
With a slight abuse of notation we let $\Theta_j$ refer to the matrix representation of $\Theta_j$ in $\RR^{r \times r}$ with respect to the basis $\cB$ for $j=1, \ldots, d$. The matrix $\Theta_j^T \in \CC^{r \times r}$ will be the usual transpose of $\Theta_j$ and let
\begin{equation}
\label{eq:Zeta1}
\xi^{z^{(k)}} := \col( p^{(q)}(z^{(k)}) )_{q=1}^{\ell}  \in \CC^r.
\end{equation}

\begin{thm}
\label{thm:June20t1}
If $s = (s_{\gamma})_{\gamma \in \NN_0^d}$ is a given real-valued multisequence such that 
$r:= {\rm rank} \, M(\infty) < \infty$ and $\xi^{z^{(k)}}$ are as in {\rm \eqref{eq:Zeta1}}, then the following facts hold{\rm :}
\begin{enumerate}
\item[(i)] $\xi^{z^{(1)}}, \ldots, \xi^{z^{(\ell)}}$ are linearly independent in $\CC^r$.
\item[(ii)] $\sigma(\Theta_j) = \sigma (\Theta_j^T) = \{ \pi_j(z^{(1)}), \ldots, \pi_j(z^{(\ell)})  \}$. Moreover,
\begin{equation}
\label{eq:June20d7}
\text{$\Theta_j^T \xi^{z^{(k)}} = \pi_j(z^{(k)}) \xi^{z^{(k)}}$ for $k=1,\ldots,r$ and $j=1,\ldots, d$.}
\end{equation}
\end{enumerate}
If, in addition, $\Theta_j$ is diagonalisable for $j=1,\ldots,d$, then
\begin{equation}
\label{eq:June20d6}
\text{$\card \cV(\cI) = r$, i.e., $\ell =r$.} 
\end{equation}
\end{thm}

\begin{proof}
A proof for assertion (i) can be found shortly before the statement of Theorem 2.9 in \cite{Laurent}. Assertion (ii) is the Stickelberger Eigenvalue Theorem (see, e.g., Theorem 2.9 in \cite{Laurent}). Let us now prove \eqref{eq:June20d6}. Since
$$\text{$\Theta_j \in \CC^{r \times r}$ is diagonalisable $\Longleftrightarrow  \Theta_j^T \in \CC^{r \times r}$ is diagonalisable},$$
the eigenvectors $\xi^{z^{(1)}}, \ldots, \xi^{z^{(\ell)}} \in \CC^r$ must satisfy
$${\rm span} \{ \xi^{z^{(1)}}, \ldots, \xi^{z^{(\ell)}}\} = \CC^r,$$
i.e., \eqref{eq:June20d6} holds.
\end{proof}

\begin{lem}
\label{lem:June21ef3}
Let $s = (s_{\gamma} )_{\gamma \in \NN_0^d}$ be a given real-valued multisequence and suppose $r := \rank M(\infty) < \infty$. Let $\cB := (x^{\lambda^{(a)}} + \cI)_{a=1}^r$ be a basis for the $r$-dimensional space $\cA := \RR[x_1, \ldots, x_d] / \cI$ {\rm (}see Lemma {\rm \ref{lem:Oct28isomorph0}}{\rm )} and write arbitrary polynomials $p+ \cI, q + \cI \in \cA$ with respect to $\cB$, i.e., 
$$
\text{$p(x) = \sum_{a=1}^r c_a x^{\lambda^{(a)}}$ and $q(x) = \sum_{b=1}^r d_b x^{\lambda^{(b)}}$.}
$$
Then
\begin{equation}
\label{eq:June21p4}
(p+\cI, q+\cI)_{\cA} = \langle M_{\cB}(\infty) {\bf p}, {\bf q} \rangle,
\end{equation}
where $M_{\cB}(\infty) \in \RR^{r \times r}$ denotes the invertible principal submatrix of $M(\infty)$ with rows and columns indexed by $\cB$, ${\bf p} = \col(c_a)_{a=1}^r \in \RR^r$ and ${\bf q} = \col(d_a)_{a=1}^r \in \RR^r$.
\end{lem}

\begin{proof}
Formula \eqref{eq:June21p4} follows from
\begin{align*}
(p+ \cI, q+\cI)_{\cA} =& \; \left( \sum_{a=1}^r c_a x^{\lambda^{(a)}}, \sum_{b=1}^r d_b x^{\lambda^{(b)}} \right)_s \\
=& \; \sum_{a,b=1}^r c_a d_b (x^{ \lambda^{(a)}}, x^{ \lambda^{(b)}} )_s \\
=& \; \sum_{a,b=1}^r c_a d_b s_{\lambda^{(a)} + \lambda^{(b)} } \\
=& \; \langle M_{\cB}(\infty) {\bf p}, {\bf q} \rangle.
\end{align*}
\end{proof}

\begin{defn}
\label{def:SUB}
Corresponding to any finite subset $\cB$ of $\NN_0^d$, let 
$$\cB + {\rm e}_j := \{ \lambda + {\rm e}_j: \lambda \in \cB \} \quad \quad {\rm for} \quad j=1,\ldots,d,$$
$M_{\cB}(\infty)$ denote the principal submatrix of $M(\infty)$ with rows and columns indexed by $\cB$ and $M_{\cB,\cB + {\rm e}_j}(\infty)$ be the submatrix of $M(\infty)$ with rows indexed by $\cB$ and columns indexed by $\cB + {\rm e}_j$
\end{defn}

\begin{lem}
\label{lem:July18prelim1}
Let $s = (s_{\gamma} )_{\gamma \in \NN_0^d}$ be a given multisequence such that
$$r := \rank M(\infty) < \infty.$$
Let $\cB = \{ X^{\lambda^{(a)} } \}_{a=1}^r$ be a basis for $\cC_{M(\infty)}$. If $s$ has an $r$-atomic quasi-complex representing measure $\mu = \sum_{a=1}^r \varrho_a \delta_{z^{(a)}}$, The following factorisations hold{\rm :}
\begin{equation}
\label{eq:FF1}
M_{\cB}(\infty) = V_{\cB}^T R V_{\cB}
\end{equation}
and
\begin{equation}
\label{eq:FF2}
M_{\cB, \cB + {\rm e}_j } = V_{\cB}^T R Z^{(j)} V_{\cB} \quad\quad {\rm for} \quad j=1, \ldots, d,
\end{equation}
where $V_{\cB}$ is the Vandermonde matrix based on $z^{(1)}, \ldots, z^{(r)}$ and $\cB$ {\rm (}see Definition \ref{def:VM}{\rm )}, $R = {\rm diag}(\varrho_1, \ldots, \varrho_r)$ and $Z^{(j)} = {\rm diag}( (z^{(1)})^{ {\rm e}_j } , \ldots, (z^{(r)})^{ {\rm e}_j } )$ for $j=1,\ldots, d$. Moreover, $V_{\cB}$ is invertible.
\end{lem}

\begin{proof}
The factorisations \eqref{eq:FF1} and \eqref{eq:FF2} are both an immediate consequence of 
$$s_{\gamma} = \sum_{a=1}^r \varrho_j (z^{(a)})^{\gamma} \quad\quad {\rm for} \quad \gamma \in \NN_0^d.$$
Since $M_{\cB}(\infty)$ is invertible and $V_{\cB}$ is a square matrix, $V_{\cB}$ must also be invertible.  
\end{proof}

\begin{rem}
\label{rem:WEIGHTS}
Since $V_{\cB} \, (1 \; 0 \; \cdots \; 0 )^T = (1 \; \cdots \; 1 )^T$, an immediate consequence of \eqref{eq:FF1} is
\begin{equation}
\label{eq:WEIGHTS}
\begin{pmatrix} \varrho_1 \\ \vdots \\ \varrho_r \end{pmatrix} = V_{\cB}^{-T} M_{\cB}(\infty) \begin{pmatrix} 1 \\ 0 \\ \vdots \\ 0 \end{pmatrix}.
\end{equation}
\end{rem}

In what follows we shall let assume $\cC_{M(\infty)}$ is finite dimensional and let $\cB = \{ X^{\lambda^{(a)} } \}_{a=1}^r$ be a basis for $\cC_{M(\infty)}$, $M_{\cB}(\infty)$ be as in Definition \ref{def:SUB} and $\Theta_j$ and $\cV(\cI)$ be as in \eqref{eq:Feb15shift} and Lemma {\rm \ref{lem:Oct28isomorph0}}, respectively.

\begin{thm}
\label{thm:Oct28fas}
Suppose $s = (s_{\gamma})_{\gamma \in \NN_0^d}$ is a given real-valued multisequence such that
$$r := {\rm rank} \, M(\infty)  < \infty.$$
If $\Theta_j$ is diagonalisable for $j=1,\ldots, d$, then there exists an $r$-atomic quasi-complex representing measure $\mu = \sum_{j=1}^r \varrho_a \delta_{z^{(a)}}$ for $s$ such that
\begin{equation}
\label{eq:Feb15o2}
\supp \mu =  \cV(\cI), 
\end{equation}
where the weights $\varrho_1, \ldots, \varrho_r$ can be computed via \eqref{eq:WEIGHTS}. If, in addition, 
\begin{equation}
\label{eq:Feb18v1}
\cV(\cI) \subseteq \RR^d,
\end{equation}
then $\mu$ is an $r$-atomic signed representing measure for $s$ whose Jordan decomposition $\mu = \mu_+ - \mu_-$ obeys
\begin{equation}
\label{eq:Hahn1}
\card \supp \mu_{\pm} = i_{\pm}( M_{\cB}(\infty) ).
\end{equation}
Moreover, if $\mathscr{M}_{s, < \infty}$ denotes the set of all finitely atomic quasi-complex representing measures for $s$, then
$$\mathscr{M}_{s, < \infty} = \{ \mu \}.$$	
\end{thm}

\begin{proof}
If $s_{\gamma} = 0$ for all $\gamma \in \NN_0^d$, then $r= 0$ and $\mu  =0$ is a representing measure for $s$ with $\card \supp \mu = 0$. Suppose $s_{\gamma}$ is not identically $0$, i.e., $r > 0$. Let $\ell := \card \cV(\cI)$ and write $\cV(\cI) = \{ z^{(1)}, \ldots, z^{(\ell)} \}$. Since 
\begin{align*}
(\Theta_j \Theta_k(p + \cI), q+\cI)_{\cA} =& \; (\Theta_j (x^{ { \rm e}_k} p + \cI), q+\cI)_{\cA} = ( x^{ {\rm e}_j+ {\rm e}_k} p + \cI, q+\cI)_{\cA}  \\
=& \; (\Theta_k \Theta_j(p + \cI), q+\cI)_{\cA} \quad {\rm for} \quad j,k=1,\ldots,d
\end{align*}
and
\begin{align*}
(\Theta_j(p+ \cI), q + \cI)_{\cA}  =& \; ( (x^{ {\rm e}_j} p)(E), q(E) )_s \\
=& \; (p(E), x^{ {\rm e}_j } q(E) )_s  \\
=& \; ( p + \cI, \Theta_j(q+ \cI) )_{\cA}  \quad \quad {\rm for} \quad p+\cI, q+\cI \in \cA,
\end{align*}
$\{ \Theta_j \}_{j=1}^d$ are commuting self-adjoint operators on the $r$-dimensional Pontryagin space with $\kappa$-negative squares $\cA$ generated by $s$ (see Lemma \ref{lem:Oct28isomorph0}).


If $\cB:= \{ X^{\lambda^{(j)}} \}_{j=1}^r$ is a basis for $\cC_{M(\infty)}$, then $\cB_{\cA} := \{ x^{\lambda^{(j)} } + \cI \}_{j=1}^r$ is a basis for $\cA$. Since $\Theta_j$ is self-adjoint on $\cA$, a complex number $\lambda \in \sigma(\Theta_j)$ if and only if $\bar{\lambda} \in \sigma(\Theta_j)$, see, e.g., Proposition 4.2.3 in \cite{GLR}. It follows from Lemma \ref{lem:July18prelim1} that the principal submatrix $M_{\cB}(\infty)$ of $M(\infty)$ with rows and columns indexed by $\cB$ is invertible. With a slight abuse of notation, we shall also use $\Theta_j$ for the matrix representation of the linear transformation $\Theta_j : \cA \to \cA$ with respect to the basis $\cB_{\cA}$. It follows from \eqref{eq:June20d6} and \eqref{eq:June20d7} that $\ell  =r$ and
$$\sigma(\Theta_j) = \{ \pi_j(z^{(1)}), \ldots, \pi_j(z^{(\ell)}) \} \quad \quad {\rm for} \quad j=1,\ldots,d.$$
Since $\Theta_j$ is diagonalisable and $\Theta_j \Theta_k = \Theta_k \Theta_j$ for $j,k=1,\ldots,d$, we have the existence of an invertible matrix $S \in \CC^{r \times r}$ such that
$$\Theta_j = S \, {\rm diag}(\pi_j(z^{(1)}), \ldots, \pi_j(z^{(r)})) S^{-1} \quad \quad {\rm for} \quad j=1,\ldots, d.$$

%
%
%
%


%
%
which is equivalent to \eqref{eq:WEIGHTS}.

If we let $\gamma = (\gamma_1, \ldots, \gamma_d) \in \NN_0^d$, $v = (1, 0, \ldots, 0)^T \in \RR^r$ and $E_{aa} \in \RR^r$ be the matrix of zeros with a $1$ in the row and column indexed by $a$, then we may use Formula \eqref{eq:June21p4} to obtain
\begin{align*}
s_{\gamma} 
=& \; (\Theta_1^{ \gamma_1} \cdots \Theta_d^{\gamma_d} (1+ \cI), 1+ \cI)_{\cA} \\
=& \; \sum_{j=1}^r \langle M_{\cB}(\infty) S \, {\rm diag}(\pi_1(z^{(j)})^{\gamma_1} \cdots \pi_d(z^{(j)})^{\gamma_d})_{j=1}^r \, S^{-1} v, v \rangle \\
=& \; \pi_1(z^{(j)})^{\gamma_1} \cdots \pi_d(z^{(j)})^{\gamma_d} \sum_{a=1}^r \langle M_{\cB}(\infty) S \, E_{aa} S^{-1} v, v \rangle \\
=& \; \sum_{a=1}^r (z^{(j)})^{\gamma} \varrho_a \quad {\rm for} \quad \varrho_a = \langle M_{\cB}(\infty) S \, E_{aa} S^{-1} v, v \rangle.
\end{align*}
Thus, $\mu = \sum_{a=1}^r \varrho_a \delta_{z^{(a)}}$ is a quasi-complex representing measure for $s$ with at most $r$ atoms and $\supp \mu \subseteq \cV(\cI)$. If we consider the truncated multisequence $\tilde{s} = (s_{\gamma})_{0 \leq |\gamma| \leq 2 n_0}$, where
$$n_0 := \max_{a =1, \ldots, r} |\lambda^{(q)}|,$$
then $\mu$ is automatically a quasi-complex representing measure for $\tilde{s}$ and
$$\rank M_{\tilde{s}}(n_0) = r,$$
whence Lemma \ref{lem:Feb22supp} forces $\card \supp \mu = r$ and hence $\supp \mu = \cV(\cI)$.


If \eqref{eq:Feb18v1} is in force, then obviously $\mu$ is a signed measure. Moreover, the factorisation \eqref{eq:FF1} and Sylvester's law of inertia (see, e.g.,  Theorem 4.5.8 in \cite{HJ}) imply that
$$
i_{\pm}( M_{\cB}(\infty)) = i_{\pm}(R), 
$$
i.e., \eqref{eq:Hahn1} holds.

We will now show that $\mathscr{M}_{s, < \infty} = \{ \mu \}$. If $r = 0$, i.e., $s_{\gamma} = 0$ for all $\gamma \in \NN_0^d$, then let $\nu \in \mathscr{M}_{s, < \infty}$ and suppose
$$\nu = \nu^{(1)}_+ - \nu^{(1)}_- + (\nu^{(2)}_+ - \nu^{(2)}_-)i,$$
is the Jordan decomposition for the quasicomplex measure $\nu$ (see, e.g., Chapter 6 in \cite{Rudin}),
where
$\nu_{\pm}^{(b)}$ for $b=1,2$ are positive finitely atomic measures on $\CC^d$ such that
$$\supp \nu^{(b)}_+ \cap \supp \nu^{(b)}_- = \emptyset \quad {\rm for} \quad b=1,2.$$
Consequently, $s_{\gamma} = 0$ for all $\gamma \in \NN_0^d$ forces
$$\int_{\CC^d} z^{\gamma} d\nu^{(b)}_+(z) = \int_{\CC^d} z^{\gamma} d\nu_-^{(b)}(z) \quad {\rm for} \quad b = 1,2.$$
But since $\card \supp \nu_{\pm}^{(b)} < \infty$, the uniqueness assertion in \cite[Corollary 2.6]{CF_several}  forces $\nu_+^{(b)} = \nu_-^{(b)}$ for $b=1,2$, whence $\nu$ is the trivial measure.

If $r > 0$, then suppose $\nu$ is an $r_0$-atomic quasi-complex representing measure for $s$ with $r_0 \geq r = \card \supp \mu$ constructed above. Then consider $\tilde{\mu} = \mu - \nu$ which is a quasi-complex representing measure for $\tilde{s} = (\tilde{s}_{\gamma})_{\gamma \in \NN_0^d}$, where $\tilde{s}_{\gamma} = 0$ for all $\gamma \in \NN_0^d$. The argument above leads one to conclude that $\mu = \nu$. Thus, we have the desired uniqueness.
\end{proof}

\begin{rem}
\label{rem:Feb23nu}
Given a real-valued multisequence $s = (s_{\gamma})_{\gamma \in \NN_0^d}$ with 
$$r := \rank M(\infty)< \infty,$$ then the set $\mathscr{M}_s$ of quasi-complex measures is always a strict superset of the $r$-atomic measure $\mu$ constructed in the proof of Theorem \ref{thm:Oct28fas}. In other words, $\mu$ is not unique outside of $\mathscr{M}_{s, < \infty}$. Indeed, if we let $\mu$ be the $r$-atomic signed representing measure for $s$ constructed in the proof of Theorem \ref{thm:Oct28fas}, then $\mu + \nu$ is also a signed representing measure for $s$, where $\nu$ is the signed measure on $\RR^d$ given by
$$
d\nu(x) = d\nu(x_1, \ldots, x_d) = \sin( x_1^{1/4}) e^{-x_1^{1/4}} dx_1 \cdots dx_d,
$$
since
$$
\int \cdots \int_{\RR^d} x_1^{\gamma_1} \cdots x_d^{\gamma_d} d\nu(x_1, \ldots, x_d) = \int_{\RR} x_1^{\gamma_1}  \sin( x_1^{1/4}) e^{-x_1^{1/4}} dx_1 = 0
$$
for all $\gamma = (\gamma_1, \ldots, \gamma_d) \in \NN_0^d$ and hence
$$\int_{\RR^d} x^{\gamma} \{d\mu(x)+d\nu(x)\} = \int_{\RR^d} x^{\gamma} d\mu(x) \quad {\rm for} \quad \gamma = (\gamma_1, \ldots, \gamma_d) \in \NN_0^d.$$
The measure $\nu$ is an analogue of a measure that is originally due to Stieljes \cite[p. J105]{Stieltjes} (see also Remark 2.7.3 in \cite{BW} for additional discussion).
\end{rem}


%
%
%
%

\begin{rem}[Construction of $\mu$ in Theorem \ref{thm:Oct28fas}] 
\label{rem:CONSTRUCTIONmu}
Let 
$$\cB = (X^{\lambda})_{\lambda \in \Lambda^d_{2n} },$$ 
with $\card \Lambda = \rank M(2n) = r$ and $M_{\cB}(n)$ be the principal submatrix of $M(n)$ with rows and columns indexed by $\cB$ and $M_{\cB, \cB+{\rm e}_j}(n+1)$ be the submatrix of $M(n+1)$ with rows indexed by $\cB$ and columns indexed by $\cB+{\rm e}_j := \{ \lambda +{\rm e}_j: \lambda  \in \cB \}$. Simultaneously diagonialise $M_{\cB}(n)^{-1} M_{\cB, \cB+{\rm e}_j }(n+1)$ obtain an invertible matrix $S$ and diagonal matrices
$$D_j = {\rm diag}(z_j^{(1)}, \ldots, z_j^{(r)} ) \quad\quad {\rm for} \quad j=1,\ldots,d$$
such that
$$M_{\cB}(n)^{-1} M_{\cB, \cB+{\rm e}_j }(n+1) = S D_j S^{-1} \quad\quad {\rm for} \quad j=1,\ldots, d.$$
Finally, let $\mu = \sum_{a=1}^r \varrho_a \delta_{z^{(a)}}$, where
\begin{equation}
\label{eq:Z}
z^{(a)} = (z_1^{(a)}, \ldots, z_d^{(a)} ) \quad\quad {\rm for} \quad a=1,\ldots, r
\end{equation}
and
\begin{equation}
\label{eq:RHO}
(\varrho_1, \ldots, \varrho_r)^T = (V^{-1})^T \col(s_{\lambda})_{\lambda \in \Lambda}
\end{equation}
where $V$ is the Vandermonde matrix with rows indexed by $z^{(1)}, \ldots, z^{(r)}$ and columns indexed by $\Lambda$.
\end{rem}

\section{Rank preserving extensions of $d$-Hankel matrices}
\label{sec:RPE}
In this section we will prove two useful technical lemmas on consistency of extensions of $d$-Hankel matrices which will be useful key for demonstrating the main result in Section \ref{sec:CFana}.

\begin{defn}
\label{def:Feb23exten}
Let $s = (s_{\gamma})_{\stackrel{0 \leq |\gamma| \leq 2n}{\gamma \in \NN_0^d}}$ be a given real-valued multisequence and $M(n) := M(n)$ be as in Definition {\rm \ref{def:Feb13dH}}. We say that $M(n)$ admits a {\it rank preserving extension} $M(n+1)$ if there exist real numbers $(s_{\gamma})_{2n+1 \leq |\gamma| \leq 2n+2}$ such that $\rank M(n+1) = \rank M(n).$
\end{defn}

\begin{lem}
\label{lem:July19k1}
Let $(s_{\gamma})_{0 \leq |\gamma| \leq 2n'}$ be a given real-valued truncated multisequence. Suppose that $p(X) = {\bf 0} \in \cC_{M(n')}$ for some $p \in \RR^d_{n'}[{\bf x}]$. If $M(n')$ has an extension $M(n'+1)$ that satisfies
$$
(x^{ {\rm e}_j } p)(X) = {\bf 0} \in \cC_{M(n'+1)} \quad \quad {\rm for} \quad j=1,\ldots,d,
$$
then 
\begin{equation}
\label{eq:July19mnprime}
p(X) = {\bf 0} \in \cC_{M(n'+1)}.
\end{equation}
\end{lem}

\begin{proof}
If $p(x) = \sum_{0 \leq |\lambda| \leq n'} p_{\lambda} x^{\lambda}$, then $p(X) = {\bf 0} \in \cC_{M(n')}$ is equivalent to
\begin{equation}
\label{eq:July19y1}
\sum_{0 \leq |\lambda| \leq n'} p_{\lambda} \col(s_{\gamma + \lambda} )_{0 \leq |\gamma| \leq n'} = {\bf 0} \in \cC_{M(n')}.
\end{equation}
Similarly,
$$
(x^{ {\rm e}_j } p)(X) = {\bf 0} \in \cC_{M(n'+1)} \quad \quad {\rm for} \quad j=1,\ldots,d,
$$
is equivalent to
\begin{equation}
\label{eq:July19y2}
\sum_{0 \leq |\lambda| \leq n'} p_{\lambda} \col(s_{\gamma + \lambda + {\rm e}_j} )_{0 \leq |\gamma| \leq n'+1} = {\bf 0} \in \cC_{M(n'+1)} \quad\quad {\rm for} \quad j=1,\ldots,d.
\end{equation}
We will use \eqref{eq:July19y1} and \eqref{eq:July19y2} to verify \eqref{eq:July19mnprime}, or, equivalently,
\begin{equation}
\label{eq:July19y3}
\sum_{0 \leq |\lambda| \leq n'} p_{\lambda} \col(s_{\eta + \lambda} )_{0 \leq |\eta| \leq n'+1} = {\bf 0} \in \cC_{M(n'+1)} \quad\quad {\rm for} \quad j=1,\ldots,d.
\end{equation}
Indeed, using \eqref{eq:July19y1} we have that \eqref{eq:July19y3} holds when $0 \leq |\eta| \leq n'$. If $\eta = (\eta_1, \ldots ,\eta_d)$ with $|\eta| = n'+1$, then $\eta_{j_0} \geq 1$ for some $j_0 \in \{ 1, \ldots, d \}$. Consequently, $gamma := \eta - {\rm e}_{j_0} \in \NN_0^d$ and $|\gamma| = n'$. Thus, we may use \eqref{eq:July19y2} to verify \eqref{eq:July19y3} when $|\eta| = n'+1$.
\end{proof}

\begin{thm}
\label{thm:June14k1}
Let $s = (s_{\gamma})_{\stackrel{0 \leq |\gamma| \leq 2n'}{\gamma \in \NN_0^d}}$ be a given real-valued truncated multisequence. If $M(n')$ admits a rank preserving extension $M(n'+1)$, then there exist a sequence of rank preserving extensions
$$(M(n'+k))_{k=2}^{\infty}$$
that gives rise to the sequence $(s_{\gamma})_{\gamma \in \NN_0^d}$ with the property
$$\rank M(\infty) = \rank M(n).$$
\end{thm}

\begin{proof}
We will provide a proof in the particular case that $d=2$. The more general case of $d > 2$ can be proved in a similar manner, albeit with more bookkeeping. For convenience we shall use $X$ and $Y$ in place of $X_1$ and $X_2$, respectively.

Since $M(n')$ has a rank preserving extension $M(n'+1)$, we have the existence of polynomials $p^{(c,d)}(x,y) = \sum_{0 \leq j+k \leq n} p_{jk}^{(c,d)} x^j y^k \in \RR_n^2[ {\bf x} ]$ such that 
\begin{equation}
\label{eq:Mar3nn2}
p^{(ab)}(X,Y) = X^a Y^b \in \cC_{M(n'+1)} \quad \text{for all $(a,b) \in \NN_0^2$ with $a + b = n'+1$.}
\end{equation}

We will proceed in a number of steps to check that a rank preserving extension $M(n'+2)$ exists. The argument can then be iterated to produce a multisequence $(s_{\alpha \beta})_{(\alpha,\beta) \in \NN_0^2}$ such that $\rank M(\infty) = \rank M(n)$.

%
%
%
%

\bigskip

\noindent {\bf Step 1:} {\it Show that}
$$(x p^{(n'+1-m, m)} )(X,Y) = (y p^{(n'+1-m+1, m-1)})(X,Y) \quad {\rm for} \quad m=1,\ldots,n'+1.$$

\bigskip

Since
$$(xp^{(n'+1-m,m)}(X,Y) = \col ( \sum_{0 \leq j+k \leq n'} p_{jk}^{(n'+1-m, m)} s_{j+c+1, k+d})_{0 \leq c+d \leq n'+1}$$
and
$$(yp^{(n'+1-m+1,m-1)}(X,Y) = \col( \sum_{0 \leq j+k \leq n'} p_{jk}^{(n'+1-m+1, m-1)} s_{j+c, k+d+1})_{0 \leq c+d \leq n'+1}$$
it suffices to show that
\begin{equation}
\label{eq:Mar3jj3}
\sum_{0 \leq j+k \leq n'} p_{jk}^{(n'+1-m,m)} s_{j+c+1, k+d} = \sum_{0 \leq j+k \leq n'} p_{jk}^{(n'+1-m+1,m-1)} s_{j+c,k+d+1}
\end{equation}
for all $0 \leq c+d \leq n'+1$.

We will check \eqref{eq:Mar3jj3} by a case analysis. 

\bigskip

\noindent {\bf Case 1a:} \eqref{eq:Mar3jj3} {\it holds for all} $(c,d) \in \NN_0^2$ {\it such that} $0 \leq c+d \leq n'$.

\bigskip

If $0 \leq c+d \leq n'$, then
$$j+c+1+k+d \leq j+k+1+n' \leq 2n'+1.$$
Thus, $s_{j+c+1,k+d}$ and $s_{j+c,k+d+1}$ are entries in $M(n'+1)$. Consequently, \eqref{eq:Mar3nn2} with $a = n'+1-m$ and $b= m$ implies that 
$$s_{n'+2-m+c,m+d} = \sum_{0 \leq j+k \leq n'} p_{jk}^{(n'+1-m, m)} s_{j+c+1,k+d} \quad {\rm for} \quad 0 \leq c+d \leq n'.$$
Similarly, \eqref{eq:Mar3nn2} with $a = n'+2-m$ and $b=m-1$ implies that
$$s_{n'+2-m+c,m-d} = \sum_{0 \leq j+k \leq n'} p_{jk}^{ (n'+2-m, m-1)} s_{j+c, k+d+1} \quad {\rm for} \quad 0 \leq c+d \leq n'.$$
Thus, \eqref{eq:Mar3jj3} holds for $0 \leq c+d \leq n'$.

\bigskip

\noindent {\bf Case 1b:} \eqref{eq:Mar3jj3} {\it holds for all} $(c,d) \in \NN_0^2$ {\it such that} $c+d = n'+1$.

\bigskip

If $c+d = n'+1$, then let
$$\omega_1 := \sum_{0 \leq j+k \leq n'} p_{jk}^{(n'+1-m,m)} s_{j+c+1, k+d}$$
and
$$\omega_2 := \sum_{0 \leq j+k \leq n'} p_{jk}^{(n'+1-m+1,m-1)} s_{j+c,k+d+1}.$$
But then
$$\omega_1 = \row ( p_{jk}^{(n'+1-m, m)} )_{0 \leq j+k \leq n'} \; \col \left( s_{j+c+1, k+d } \right)_{0 \leq j+k \leq n'}.$$
Since $\col \left( s_{j+c+1, k+d} \right)_{0 \leq j+k \leq n'}$ is a subvector of $X^c Y^d \in \cC_{M(n'+1)}$. Consequently, $\rank M(n'+1) = \rank M(n')$ forces 
$$\col \left( s_{j+c+1, k+d} \right)_{0 \leq j+k \leq n'} = \Phi \, \col ( p_{jk}^{(c,d)} )_{0 \leq j+k \leq n'},$$
where $\Phi \in \RR^{ \ell(n',2) \times \ell(n',2) }$ is the submatrix of $M(n'+1)$ with rows indexed by $(X^{j+1} Y^k)_{0 \leq j+k \leq n'}$ and columns indexed by $(X^j Y^k)_{0 \leq j +k \leq n'}$. Thus,
\begin{equation}
\label{eq:Mar3ok1}
\omega_1 = \row ( p_{jk}^{(n'+1-m, m)} )_{0 \leq j+k \leq n'} \, \Phi \, \col ( p_{jk}^{(c,d)} )_{0 \leq j+k \leq n'}.
\end{equation}

On the other hand,
$$\omega_2 = \row (p_{jk}^{(n'+1-m+1,m-1)})_{0 \leq j+k \leq n'} \; \col (s_{j+c, k+d+1} )_{0 \leq j+k \leq n'}.$$
But $\col (s_{j+c, k+d+1})_{0 \leq j+k \leq n'}$ is a subvector of $X^c X^d \in \cC_{M(n'+1)}$. Consequently, $\rank M(n'+1) = \rank M(n')$ implies that 
$$\col \left( s_{j+c, k+d+1} \right)_{0 \leq j+k \leq n'} = \wt{\Phi} \, \col ( p_{jk}^{(c,d)} )_{0 \leq j+k \leq n'},$$
where $\wt{\Phi} \in \RR^{ \ell(n',2) \times \ell(n',2) }$ is the submatrix of $M(n'+1)$ with rows indexed by $(X^{j} Y^{k+1})_{0 \leq j+k \leq n'}$ and columns indexed by $(X^j Y^k)_{0 \leq j +k \leq n'}$. Thus,
\begin{equation}
\label{eq:Mar3ok2}
\omega_2 = \row ( p_{jk}^{(n'+1-m+1, m-1)} )_{0 \leq j+k \leq n'} \, \wt{\Phi} \, \col ( p_{jk}^{(c,d)} )_{0 \leq j+k \leq n'}.
\end{equation}

We claim that $\Phi$ and $\wt{\Phi}$ are Hermitian. Indeed, the entry in the row indexed by $(j+1, k)$ and the column indexed by $(\ell, m)$ is given by
$$s_{j+1+\ell, k+m},$$
 while the entry in the row indexed by the column $(\ell +1, m)$ and the column indexed by $(j,k)$ is given by 
$$s_{\ell + j + 1, m +k}.$$ 
Thus, $\Phi$ is Hermitian. The fact that $\wt{\Phi}$ is Hermitian can be justified in much the same way.

Since $\Phi$ is Hermitian, \eqref{eq:Mar3ok1} and \eqref{eq:Mar3nn2} imply that
\begin{align*}
\omega_1 =& \;  \row(p_{jk}^{(c,d)})_{0 \leq j+k \leq n'} \, \Phi \, \col (p^{(n'+1-m, m)}_{jk})_{0 \leq j+k \leq n'}  \\
=& \; \row(p_{jk}^{(c,d)})_{0 \leq j+k \leq n'} \, \col (s_{n'+1-m+j+1, m+k} )_{0 \leq j+k \leq n'}
\end{align*}
Similarly, since $\wt{\Phi}$ is Hermitian, \eqref{eq:Mar3ok2} and \eqref{eq:Mar3nn2} imply that
\begin{align*}
\omega_2 =& \;  \row(p_{jk}^{(c,d)})_{0 \leq j+k \leq n'} \, \wt{\Phi} \, \col (p^{(n'+1-m+1, m-1)}_{jk})_{0 \leq j+k \leq n'}  \\
=& \; \row(p_{jk}^{(c,d)})_{0 \leq j+k \leq n'} \, \col (s_{n'+1-m+j+1, m-1+k+1} )_{0 \leq j+k \leq n'}
\end{align*}
Thus, $\omega_1 = \omega_2$ and we have shown \eqref{eq:Mar3jj3} when $c+d = n'+1$.

\bigskip

\noindent {\bf Step 2:} {\it Define matrices $B \in \RR^{ \ell(n',1) \times \ell(n'+2,2)}$ and $C \in \RR^{\ell(n'+2,2) \times \ell(n'+2,2)}$ such that the block matrix
$$M := \begin{pmatrix} M(n+1) & B \\ B^T & C \end{pmatrix}$$
satisfies $\rank M = \rank M(n'+1)$.}

\bigskip

To this end, let
$$B = \row(v^{(c,d)} )_{c+d = n'+2},$$
where $v^{(c,d)} \in \RR^{\ell(n'+1,2)}$ given by
\begin{align}
v^{(c,d)} :=& \; \col( v_{\alpha \beta}^{(c,d)} )_{0 \leq \alpha+\beta \leq n'+1} \nonumber \\
:=& \; \begin{cases}
\sum_{ 0 \leq j+k \leq n} p_{jk}^{(c-1,d)} s_{j+\alpha+1, \beta + k} & {\rm if} \quad c > 0, \\[6pt]
\sum_{0 \leq j+k \leq n} p_{jk}^{(0, d-1)} s_{j+\alpha, k+\beta+1} & {\rm if} \quad c = 0. 
\end{cases}
\label{eq:Mar4k1}
\end{align}
Thus, 
$$B = \row\{ (x p^{(n'+1,0)})(X,Y), (x p^{(n',1)})(X,Y), \ldots, (x p^{(1,n')})(X,Y), (y p^{(0,n'+1)})(X,Y) \}$$
and consequently there exists a matrix $W \in \RR^{ \ell(n'+1, 2) \times \ell(n'+2, 2)}$ such that
$$M(n'+1) W = B.$$
If we let 
\begin{equation}
\label{eq:Mar3C}
C := W^T M(n'+1) W = B^T W,
\end{equation}
then $\rank M = \rank M(n'+1)$. This completes the proof of Step 2.

We would like to check that $M$ is $2$-Hankel. In view of Lemma \ref{lem:dH}, it suffices to show that
\begin{equation}
\label{eq:Mar32H}
(x^a y^b, x^c y^d )_M = (x^{\alpha} y^{\beta}, x^{\delta} x^{\varepsilon} )_M
\end{equation}
whenever $a+c = \alpha + \delta$ and $b + d = \beta + \varepsilon$ for
$$0 \leq a+b, c+d, \alpha+\beta, \delta+\varepsilon \leq n+2.$$
Since $M(n'+1)$ is $2$-Hankel, Lemma \ref{lem:dH} implies that \eqref{eq:Mar32H} holds whenever
$$0 \leq a+b, c+d, \alpha+\beta, \delta+\varepsilon \leq n'+1.$$

\bigskip

\noindent {\bf Step 3:} {\it Verify \eqref{eq:Mar32H} when $0 \leq a+b, \alpha + \beta \leq n'+1$ and $c+d = \delta+\varepsilon = n'+2.$}

\bigskip

We begin by claiming that if $b > 0$ and $c > 0$, then
\begin{equation}
\label{eq:Mar6c1}
\text{$(x^a y^b, x^c y^d)_M = (x^{a+1} y^{b-1}, x^{c-1} y^{d+1} )_M$ for $0 \leq a +b \leq n'+1$ and $c+d = n'+2.$}
\end{equation}
Indeed, in view of the construction of $B$ given in Step 2, we may use the fact that $M(n+1)$ is $2$-Hankel (and hence Lemma \ref{lem:dH} is applicable) to obtain
\begin{align*}
(x^a y^b, x^c y^d )_M =& \; \left(x^a y^b, \sum_{0 \leq j+k \leq n'} p_{jk}^{(c-1, d)} x^{j+1} y^k \right)_M \\
=& \; \left( x^a y^b, x \sum_{0 \leq j+k \leq n'} p_{jk}^{(c-1, d)} x^{j+1} y^k \right)_M \\
=& \; \left( x^{a+1} y^{b-1}, \sum_{0 \leq j+k \leq n'} p_{jk}^{(c-1, d)} x^{j} y^{k+1} \right)_M \\
=& \;  \left( x^{a+1} y^{b-1}, (y p^{(c-1,d)})(X,Y) \right)_M \\
=& \;  \left( x^{a+1} y^{b-1}, x^{c-1} y^{d+1} \right)_M.
\end{align*}
Thus, \eqref{eq:Mar6c1} is verified.

We will now verify \eqref{eq:Mar32H}. Without loss of generality, suppose 
$$X^c Y^d \preceq_L X^{\delta} Y^{\varepsilon}.$$
Lemma \ref{lem:Mar5tg1} asserts that $X^{\alpha} Y^{\beta} \preceq_L X^a Y^b.$
If $\delta = c$ and $\varepsilon=d$, then $ a= \alpha$ and $\beta = b$, whence there is nothing to prove. We may assume, without loss of generality, that
$$X^c Y^d \preceq_L X^{\delta}Y^{\varepsilon}$$
and $(c,d) \neq (\delta,\varepsilon)$. Therefore, $d < \varepsilon$ and $ c > \delta
$. Assertion \eqref{eq:Mar5k3} implies that
$a + b = \alpha + \beta$, since $c + d = \delta+ \varepsilon = n'+2$. If $a = \alpha$ and $b = \beta$, then $c = \delta$ and $ d = \varepsilon$. However, this cannot happen because we assumed that $(c,d) \neq (\delta, \varepsilon)$. Thus,
$X^{\alpha} Y^{\beta} \preceq_L X^a Y^b$ forces 
$$a < \alpha \quad {\rm and} \quad \beta > b.$$
Since $a+c = \alpha + \delta$ and $b+d = \beta + \varepsilon$, we have 
$$\alpha - a = c - \delta \quad {\rm and} \quad b - \beta = \varepsilon - d.$$
Furthermore, since $c+d = \delta + \varepsilon =n'+2$, we have $c - \delta = \varepsilon-d$ and hence
$$\omega := \alpha - a = c- \delta = \varepsilon - d = b - \beta.$$
Therefore, we may invoke formula \eqref{eq:Mar6c1} $\omega$-times to obtain
\begin{align*}
(x^a y^b, x^c y^d )_M =& \; (x^{a+1} y^{b-1}, x^{c-1} x^{d+1} )_M \\
\vdots &  \\
=& \; (x^{\alpha-1} y^{\beta+1}, x^{\delta+1} y^{\varepsilon-1} )_M \\
=& \; (x^{\alpha} y^{\beta}, x^{\delta} y^{\varepsilon})_M.
\end{align*}
Thus, have verified \eqref{eq:Mar32H}. This completes the proof of Step 3.

In view of Step 3, we may index the rows of $B^T$ by the monomials $(X^c Y^d)_{c+d=n'+2}$ and the columns of $B^T$ by the monomials $(X^a Y^b)_{0 \leq a+b \leq n'+1}$ so that the entry in the row indexed by $X^c Y^d$ and the column indexed by $X^a X^b$ is given by
$$s_{c+a,d+b}.$$
We will let $\wt{X}^a \wt{Y}^b := \col(s_{c+a,d+b})_{c+d=n'+2} \in \RR^{n'+1}$. In Step 2, we constructed a matrix $W$ such that $M(n'+1)W = B$ and $C = W^T M(n'+1) W = B^T W.$ Thus, we have 
$$ C = \row\{ (x p^{(n'+1,0)})(\wt{X},\wt{Y}), (x p^{(n',1)})(\wt{X},\wt{Y}), \ldots, (x p^{(1,n')})(\wt{X},\wt{Y}), (y p^{(0,n'+1)})(\wt{X},\wt{Y}) \}.$$

\bigskip

\noindent {\bf Step 4:} {\it Show that $(xp^{(n'+1-m,m)})(\wt{X}, \wt{Y}) = (yp^{(n'+1-m+1,m-1)})(\wt{X},\wt{Y})$.}

\bigskip

Similar to Step 1, the equality 
$$(x p^{(n'+1-m,m)})(\wt{X}, \wt{Y}) = (y p^{(n'+1-m+1,m-1)} )(\wt{X}, \wt{Y}) \quad {\rm for} \quad m=1,\ldots n'+1$$ 
is in force if and only if
\begin{align}
& \; \; \col ( \sum_{0 \leq j+k \leq n'} p_{jk}^{(n'+1-m,m)} s_{j+c+1, k+d} )_{c+d = n'+2} \nonumber \\
=& \; \col ( \sum_{0 \leq j+k \leq n'} p_{jk}^{(n'+1-m+1, m-1)} s_{j+c, k+d+1} )_{c+d=n'+2} \label{eq:Mar6kk3}
\end{align}
for $m=1,\ldots, n'+1$.

Fix $(c,d) \in \NN_0^2$ such that $c+d = n'+2$. Then
\begin{align*}
\sigma_1 := & \; \sum_{0 \leq j+k \leq n'} p_{jk}^{(n'+1-m,m)} s_{j+c+1, k+d} \\
=& \; \row(p_{jk}^{(n'+1-m,m)} )_{0 \leq j+k \leq n'} \; \col(s_{j+c+1, k+d})_{0 \leq j+k \leq n'}.
\end{align*}
Since $\col(s_{j+c+1, k+d})_{0 \leq j+k \leq n'}$ is subvector of the column of $B$ indexed by the monomial $X^c X^d$, we have
$$\col( s_{j+c+1, k+d} )_{0 \leq j+k \leq n'} = \Phi \col(p_{jk}^{(c,d)})_{0 \leq j+k \leq n'},$$
where $\Psi \in \RR^{\ell(n',2) \times \ell(n',2)}$ is the submatrix of $M(n'+1)$ with rows indexed by $(X^{j+1} Y^k)_{0 \leq j+k \leq n'}$ and columns indexed by $(X^j Y^k)_{0 \leq j+k \leq n'}$. Similarly, 
\begin{align*}
\sigma_2 := & \; \sum_{0 \leq j+k \leq n'} p_{jk}^{(n'+1-m+1,m-1)} s_{j+c, k+d+1} \\
=& \; \row(p_{jk}^{(n'+1-m+1,m-1)} )_{0 \leq j+k \leq n'} \; \col(s_{j+c, k+d+1})_{0 \leq j+k \leq n'}.
\end{align*}
Since $\col(s_{j+c, k+d+1})_{0 \leq j+k \leq n'}$ is subvector of the column of $B$ indexed by the monomial $X^c X^d$, we have
$$\col( s_{j+c+1, k+d} )_{0 \leq j+k \leq n'} = \Phi \col(p_{jk}^{(c,d)})_{0 \leq j+k \leq n'},$$
where $\wt{\Psi} \in \RR^{\ell(n',2) \times \ell(n',2)}$ is the submatrix of $M(n'+1)$ with rows indexed by $(X^{j} Y^{k+1})_{0 \leq j+k \leq n'}$ and columns indexed by $(X^j Y^k)_{0 \leq j+k \leq n'}$.

One can verify that $\sigma_1 = \sigma_2$ in the much the same way as we verified that $\omega_1 = \omega_2$ in Step 1. This completes the proof of Step 4.

\bigskip

\noindent {\bf Step 5:} {\it Verify \eqref{eq:Mar32H} when $a+b = c+d = \alpha + \beta = \delta + \varepsilon = n'+2$}.

\bigskip

The only case in \eqref{eq:Mar32H} left to verify is when $a+b = c+d = \alpha + \beta = \delta + \varepsilon = n'+2$. However, with the presence of Step 4, the justification of Step 3 can easily be adjusted to fit the present context. 
\end{proof}

Let 
\begin{equation}
\label{eq:Jun1e1}
\Lambda_m^d = \{ \gamma \in \NN_0^d: |\gamma| = m \}
\end{equation}
and
\begin{equation}
\label{eq:Jun1e2}
\Lambda_{m,n'}^d = \{ \gamma \in \Lambda_m^d: \gamma = n'{\rm e}_{j_0} + \sum_{j=1}^d \gamma_j {\rm e}_j \quad \text{for some $j_0 \in \{ 1, \ldots, d \}$} \}.
\end{equation}
Clearly, $\Lambda_{m,n'}^d \subseteq \Lambda_m^d.$

\begin{lem}
\label{lem:Jun1p1}
If $n > 0$ and $\Lambda_m^d$ and $\Lambda_{m,n}^d$ are as in {\rm \eqref{eq:Jun1e1}} and {\rm\eqref{eq:Jun1e2}}, respectively, then
\begin{mini*}
		{}{m}
		{}{}
		\addConstraint{\Lambda^d_{m,n'}}{=\Lambda^d_m}		
	\end{mini*}
	is given by $(n-1)(d-1) + 1.$
\end{lem}

\begin{proof}
If we let $m_0$ be the minimal value of the optimisation problem in the statement of Lemma \ref{lem:Jun1p1}, then $m_0 \leq n+(n-1)(d-1)$, since $\Lambda^d_{n+(n-1)(d-1),n} = \Lambda_{n+(n-1)(d-1)}^d$. If not, then 
$$| \gamma | \leq (n-1)d < n+(n-1)(d-1) = (n-1)d+1.$$
To see that $m_0 = n+(n-1)(d-1)$, it suffices to show that if 
$$m_0 = n+(n-1)(d-1) -1 = (n-1)d,$$ 
then $\Lambda_{m_0,n}^d$ is a strict subset of $\Lambda_{m_0}^d.$ Indeed, as 
$$(n-1, \ldots, n-1) \in \NN_0^d$$
belongs to $\Lambda_{m_0}^d$ but not to $\Lambda_{m_0,n}^d$, we have that
$\Lambda_{m_0,n}^d$ is a strict subset of $\Lambda_{m_0}^d.$
\end{proof}

\begin{lem}
\label{lem:Mar28l1}
Let $s = (s_{\gamma} )_{\stackrel{0 \leq |\gamma| \leq 2n'}{\gamma \in \NN_0^d} }$ be a given multisequence of real numbers with $n' > 0$. If $M(n')$ has an extension $M(n'+1)$ such that
\begin{equation}
\label{eq:Mar28e1}
X^{(n'+1){\rm e}_j} = p^{(n'+1){\rm e}_j}(X) \in \cC_{M(n'+1)} \quad {\rm for} \quad p^{(n'+1){\rm e}_j}(x) \in \RR_n^d[ {\bf x} ]
\end{equation}
and $j=1,\ldots, d$, then $M(n'+1)$ gives rise to a sequence of extensions $(M(n'+a))_{a=2}^{\infty}$ such that for any $j_1 \in \{ 1, \ldots, d \}$,
\begin{equation}
\label{eq:Mar28kk3}
X^{(n'+1){\rm e}_{j_1} + { \rm e}_{j_2}+ \ldots + {\rm e}_{j_a}} = ( x^{ {\rm e}_{j_1} + {\rm e}_{j_2} + \ldots + {\rm e}_{j_a}} p^{(n'+1){\rm e}_{j_1} })(X) \in \cC_{M(n'+a)}
\end{equation}
for all $j_k = 1, \ldots, d$ and $k=2,\ldots, a$. In particular,
\begin{equation}
\label{eq:May15f1}
\rank M(n'd) = \rank M(n'd+1) = \rank M(\infty).
\end{equation}
\end{lem}

\begin{proof}
The sequence of extensions $(M(n'+a))_{a=2}^{\infty}$ will be generated using \eqref{eq:Mar28e1}. Let us suppose that first that \eqref{eq:Mar28kk3} is in force. Let
$$L_a = \{ \gamma \in \NN_0^d: X^{\gamma} = ( x^{ {\rm e}_{j_1} + {\rm e}_{j_2} + \ldots + {\rm e}_{j_a}} p^{(n'+1){\rm e}_{j_1} })(X) \in \cC_{M(n'+a)} \}.$$
Lemma \ref{lem:Jun1p1} with $n = n'+1$ asserts that the smallest $a$ such that $L_a \cap \Lambda_{n'+a}^d = \Lambda_{n'+a}^d$ is given by $n'(d-1) + 1$. Consequently, we can invoke Theorem \ref{thm:June14k1} to obtain \eqref{eq:May15f1}.

The proof of \eqref{eq:Mar28kk3} is broken into steps.

\bigskip

\noindent {\bf Step 1:} {\it Define extensions $(M(n'+a))_{a=2}^{\infty}$ satisfying \eqref{eq:Mar28e1} and \eqref{eq:Mar28kk3} such that $M(n'+a+1)$ obeys
$$X^{(n'+a+1){\rm e}_j} = (x^a p^{(n'+1){\rm e}_j})(X) \in \cC_{M(n'+a)} \quad {\rm for} \quad j=1,\ldots, d.$$}

\bigskip

If \eqref{eq:Mar28e1} is in force, then we shall first show that $M(n'+1)$ admits an extension $M(n'+2)$ that is completely determined by the column relations
$$X^{(n'+2){\rm e}_j } = (x^{ {\rm e}_j} p^{ (n'+1){\rm e}_j } )(X) \quad {\rm for} \quad j=1,\ldots, d.$$
Indeed, if we let
\begin{equation}
\label{eq:Mar28l1}
{\rm col}(s_{ (n'+2){\rm e}_j + \lambda})_{|\lambda| = n'+1} := {\rm col}(\sum_{0 \leq |\alpha | \leq n} p_{\alpha}^{(n'+1) {\rm e}_j } s_{\alpha+ \lambda + {\rm e}_j} )_{|\lambda| = n'+1}
\end{equation}
and
\begin{equation}
\label{eq:Mar28l2}
{\rm col}(s_{ (n'+2){\rm e}_j + \gamma})_{|\gamma| = n'+2} := {\rm col}(\sum_{0 \leq |\alpha| \leq n'} p_{\alpha}^{(n+1){\rm e}_j } s_{\alpha+\gamma+{\rm e}_j } )_{|\gamma| =n'+2},
\end{equation}
then we must first show that \eqref{eq:Mar28l1} and \eqref{eq:Mar28l2} are consistent. To this end, we only have to check that
$$\sum_{0 \leq |\alpha| \leq n'} p_{\alpha}^{(n'+1){\rm e}_j } s_{\alpha+(n'+2){\rm e}_k+{\rm e}_j } = \sum_{0 \leq |\alpha| \leq n'} p_{\alpha}^{(n'+1){\rm e}_k } s_{\alpha+{\rm e}_k+(n'+2){\rm e}_j }  \quad {\rm for} \quad j,k=1,\ldots, d.$$
Since ${\rm col}(s_{\alpha + (n'+2){\rm e}_k + {\rm e}_j} )_{0 \leq |\alpha| \leq n'}$ is subvector of $X^{(n'+1){\rm e}_k}$ with rows indexed by $(X^{\alpha+{\rm e}_k+{\rm e}_j})_{0 \leq |\alpha| \leq n'}$, \eqref{eq:Mar28e1} imples that
$${\rm col}(s_{\alpha + (n'+2){\rm e}_k + {\rm e}_j} )_{0 \leq |\alpha| \leq n'} = \Phi\,  {\rm col}(p_{\alpha}^{ (n'+1){\rm e}_k })_{0 \leq |\alpha| \leq n'},$$
where $\Phi$ is the submatrix of $M(n'+1)$ with rows indexed by $(X^{\alpha +{\rm e}_k + {\rm e}_j })_{0 \leq |\alpha|\leq n}$ and columns indexed by $(X^{\alpha})_{0 \leq |\alpha| \leq n'}$. Lemma \ref{lem:dHHerm} implies that $\Phi = \Phi^*$. Thus, 
\begin{align*}
\sum_{0 \leq |\alpha| \leq n'} p_{\alpha}^{(n'+1){\rm e}_j } s_{\alpha+(n'+2){\rm e}_k+{\rm e}_j } =& \; {\rm row}(p_{\alpha}^{(n'+1) {\rm e}_j} )_{0 \leq |\alpha|\leq n'} \; {\rm col}(s_{\alpha+{\rm e}_k+{\rm e}_j } )_{0 \leq |\alpha|\leq n'} \\
=& \; {\rm row}(p_{\alpha}^{(n+1) {\rm e}_j} )_{0 \leq |\alpha|\leq n} \, \Phi\,  {\rm col}(p_{\alpha}^{ (n'+1){\rm e}_k })_{0 \leq |\alpha| \leq n'} \\
=& \; {\rm row}(p_{\alpha}^{(n+1) {\rm e}_k} )_{0 \leq |\alpha|\leq n'} \, \Phi\,  {\rm col}(p_{\alpha}^{ (n'+1){\rm e}_j })_{0 \leq |\alpha| \leq n'} \\
=& \; {\rm row}(p_{\alpha}^{(n'+1) {\rm e}_k} )_{0 \leq |\alpha|\leq n'} \; {\rm col}(s_{\alpha + (n'+2){\rm e}_j + {\rm e}_k} )_{0 \leq |\alpha| \leq n'} \\
=& \; \sum_{0 \leq |\alpha| \leq n'} p_{\alpha}^{(n'+1){\rm e}_k } s_{\alpha+{\rm e}_k+(n'+2){\rm e}_j }.
\end{align*}

Next, we shall show that
\begin{equation}
\label{eq:Mar29t1}
\col(s_{(n'+2){\rm e}_j+\lambda } )_{0 \leq |\lambda| \leq n'} = \col( \sum_{0 \leq |\alpha | \leq n'} p_{\alpha}^{(n'+1){\rm e}_j } s_{\alpha + \lambda + {\rm e}_j })_{0 \leq |\lambda| \leq n'}.
\end{equation}
Since $\col(s_{(n'+2){\rm e}_j+\lambda } )_{0 \leq |\lambda| \leq n'}$ is a subvector of $X^{(n'+1){\rm e}_j} \in \cC_{M(n'+1)}$ with rows indexed by 
$$(X^{ {\rm e}_j + \lambda} )_{0 \leq |\lambda| \leq n'},$$
\eqref{eq:Mar28e1} implies that \eqref{eq:Mar29t1} holds.

The above argument can easily be generalised to show the existence of extensions $(M(n'+a))_{a=2}^{\infty}$ such that $M(n'+a)$ has the property
$$X^{(n'+a){\rm e}_j} = (x^a p^{(n'+1){\rm e}_j})(X) \quad {\rm for} \quad j = 1,\ldots,d.$$
This completes Step 1.

It is perhaps best to consider the case when $d=2$. The more general case when $d > 2$ follows in much the same manner, albeit with heavier notation. We will use $x, y$ and $X,Y$ in place of $x_1, x_2$ and $X_1, Y_1$, respectively. If $d =2$, then \eqref{eq:Mar28kk3} is equivalent to showing that $M(n+a)$ has an extension $M(n+a+1)$ such that
\begin{equation}
\label{eq:Mar29e1}
X^{n'+1+a-m} Y^m = (x^{a-m}y^m p^{(n'+1,0)})(X,Y) \quad {\rm for} \quad m=0, \ldots, a
\end{equation}
and
\begin{equation}
\label{eq:Mar29e2}
X^m Y^{n'+1+a-m}  = (x^m y^{a-m}p^{(0, n'+1)})(X,Y) \quad {\rm for} \quad m=0, \ldots, a.
\end{equation}

We will show \eqref{eq:Mar29e1} and \eqref{eq:Mar29e2} by induction on $a$. Suppose $M(n'+a)$ has an extension $M(n'+a+1)$ such that \eqref{eq:Mar29e1} and \eqref{eq:Mar29e2} hold. We must show that $M(n+a+1)$ has an extension $M(n'+a+2)$ such that
\begin{equation}
\label{eq:Mar29e3}
X^{n'+a+2-m} Y^m = (x^{a+1-m}y^m p^{(n'+1,0)})(X,Y) \quad {\rm for} \quad m=0, \ldots, a+1
\end{equation}
and
\begin{equation}
\label{eq:Mar29e4}
X^m Y^{n'+a+2-m}  = (x^m y^{a+1-m}p^{(0, n'+1)})(X,Y) \quad {\rm for} \quad m=0, \ldots, a+1.
\end{equation}
It is readily realised that \eqref{eq:Mar29e3} and \eqref{eq:Mar29e4} are equivalent to 
\begin{align}
& \;\col(s_{n'+a-m+1+\alpha, m+1 + \beta})_{0 \leq \alpha+ \beta \leq n'+a+2} \nonumber \\
=& \; \col(\sum_{0 \leq j+k \leq n} p_{jk}^{(n'+1,0)} s_{j+\alpha+a-m,k+\beta+m+1} )_{0 \leq \alpha+\beta \leq n'+a+2} \label{eq:Mar29e7}
\end{align}
and
\begin{align}
& \; \col( s_{m+1+\alpha,n'+a-m+1+\beta})_{0 \leq \alpha + \beta \leq n'+a+2}\nonumber \\
=& \; \col(\sum_{0 \leq j+k \leq n'} p_{jk}^{(0,n'+1)} s_{j+\alpha+m+1, k+\beta+a-m} )_{0 \leq \alpha+\beta \leq n'}, \label{eq:Mar29e8}
\end{align}
respectively.


%
We will check that \eqref{eq:Mar29e7} and \eqref{eq:Mar29e8} hold by induction on $m$ in a number of successive steps.  The case that $m = 0$ follows from the definition of the extension $M(n'+a+2)$ in Step 1. If \eqref{eq:Mar29e7} and \eqref{eq:Mar29e8} hold for $m > 0$, then we must show that
$$
X^{n'+a+2-(m+1)} Y^{m+1} = (x^{a+1-(m+1)} y^{m+1} p^{(n'+1,0)} )(X,Y) \in \cC_{M(n'+a+2)}
$$
and
$$
X^{m+1} Y^{n'+a+2-(m+1)} = (x^{m+1} y^{a+1-(m+1)} p^{(0,n'+1)})(X,Y) \in \cC_{M(n'+a+2)} 
$$
hold, i.e.,
\begin{align}
& \; \col( s_{\alpha+n'-m+1,\beta+m+1} )_{0 \leq \alpha+\beta \leq n'+a+2} \nonumber \\
=& \; \col( \sum_{0 \leq j+k \leq n'} p_{jk}^{(n'+1,0)} s_{j+\alpha+a-m, k+\beta + m + 1} )_{0 \leq \alpha + \beta \leq n'+a+2} \label{eq:Mar29p1}
\end{align}
and
\begin{align}
& \; \col( s_{\alpha+m+1,\beta+n'+a+1-m} )_{0 \leq \alpha+\beta \leq n'+a+2} \nonumber \\
=& \; \col( \sum_{0 \leq j+k \leq n'} p_{jk}^{(0,n'+1)} s_{j+\alpha+m+1, k+\beta +a- m} )_{0 \leq \alpha + \beta \leq n'+a+2}. \label{eq:Mar29p2}
\end{align}

\bigskip

\noindent {\bf Step 2:} {\it Check that \eqref{eq:Mar29p1} holds when $0 \leq \alpha + \beta \leq n'+a$.}

\bigskip

If $0 \leq \alpha + \beta \leq n'+a$, then $\col (s_{\alpha+n'+a-m+1, \beta+m+1} )_{0 \leq \alpha + \beta \leq n'+a}$ is a subvector of
$$X^{n'+a-m+1}Y^m \in \cC_{M(n'+a+1)}$$
with rows indexed by
$$(X^{\alpha} Y^{\beta+1})_{0 \leq \alpha+\beta \leq n'+a}.$$
Thus, \eqref{eq:Mar29e3} implies that
\begin{align*}
& \; \col(s_{\alpha+n'+a-m+1,\beta+m+1} )_{0 \leq \alpha+\beta \leq n+a} \\
=& \; \col(\sum_{0 \leq j+k \leq n'} p_{jk}^{(n'+1,0)} s_{\alpha+j+a-m, \beta+k+m+1} )_{0 \leq \alpha+\beta \leq n'+a},
\end{align*}
i.e., \eqref{eq:Mar29p1} holds when $0 \leq \alpha + \beta \leq n'+a$.

\bigskip

\noindent {\bf Step 3:} {\it Check that \eqref{eq:Mar29p1} holds when $(\alpha,\beta) \in \Lambda^2_{n'+a+1} \, \setminus \, \{ (0,n'+a+1) \}$.}

\bigskip

Since 
$$\col(s_{\alpha+n'+a-m+1, \beta+m+1} )_{(\alpha,\beta) \in \Lambda^2_{n'+a+1} \setminus \{ (0,n'+a+1) \} }$$
is a subvector of
$$X^{n'+a+2-m}Y^m \in \cC_{M(n'+a+2)}$$
with rows indexed by
$$(X^{\tilde{\alpha}}Y^{\tilde{\beta}})_{(\tilde{\alpha},\tilde{\beta}) \in \Lambda^2_{n'+a+1} \setminus \{ (n'+a+1,0) \}},
$$
the induction hypothesis \eqref{eq:Mar29e7} implies that
\begin{align}
\; & \col(s_{\alpha+n'+a-m+1, \beta+m+1} )_{ (\alpha, \beta) \in \Lambda^2_{n'+a+1} \setminus \{ (0,n'+a+1) \} } \nonumber \\
=& \; \col (\sum_{0 \leq j+k \leq n'} p_{jk}^{(n+1, 0)} s_{\tilde{\alpha}+j+a-m+1,\tilde{\beta}+k+m} )_{(\tilde{\alpha},\tilde{\beta}) \in \Lambda^2_{n'+a+1} \setminus \{ (n'+a+1, 0) \} }. \label{eq:Apr10e1}
\end{align}

If $(\alpha,\beta) \in \Lambda^2_{n'+a+1} \, \setminus \, \{(0,n'+a+1) \}$, then
$$\alpha = n' +a+1-u \quad {\rm and} \quad \beta=u \quad {\rm for} \quad u=0,\ldots, n'+a.$$
Similarly, if $(\tilde{\alpha},\tilde{\beta}) \in \Lambda^2_{n'+a+1} \, \setminus \, (0,n'+a+1)$, then
$$
\tilde{\alpha} = n+a+1-\tilde{u} \quad {\rm and }\quad \tilde{\beta} =  \tilde{u} \quad {\rm for} \quad \tilde{u} = 1, \ldots, n'+a+1.
$$
Thus, \eqref{eq:Apr10e1} can be rewritten in the following manner:
\begin{align*}
& \; \col( \sum_{0 \leq j+k \leq n'} p_{jk}^{(n'+1,0)} s_{\tilde{\alpha}+j+a-m+1,\tilde{\beta}+k+m})_{(\tilde{\alpha},\tilde{\beta}) \in \Lambda^2_{n'+a+1} \setminus \{ (n'+a+1, 0) \} } \\
=& \; \col( \sum_{0 \leq j+k \leq n'} p_{jk}^{(n'+1,0)} s_{\tilde{u}+j+n'+a-m+1,\tilde{u}+k+m+1}  )_{\tilde{u}=1}^{n'+a+1} \\
=& \; \col( \sum_{0 \leq j+k \leq n'} p_{jk}^{(n'+1,0)} s_{j+a-(\tilde{u}-1)-m, k+(\tilde{u}-1)+m+1} )_{\tilde{u}=1}^{n'+a+1} \\
=& \; \col(\sum_{0 \leq j+k \leq n'} p_{jk}^{(n'+1,0)} s_{j+n+a+1-u-m, k+u+m+1} )_{u=0}^{n'+a} \\
=& \; \col( \sum_{0 \leq j+k \leq n'} p_{jk}^{(n'+1,0)} s_{\alpha+j-u-m, \beta+k+m+1} )_{(\alpha,\beta) \in \Lambda^2_{n'+a+1} \setminus \{ (0,n'+a+1) \} },
\end{align*}
i.e., \eqref{eq:Mar29p1} holds when $\alpha + \beta = n'+a+1$.

\bigskip

\noindent {\bf Step 5:} {\it Check that \eqref{eq:Mar29p1} holds when $\alpha =0$ and  $\beta = n'+a+1$.}

\bigskip

We must show that 
$$
s_{n'+a-m+1, m+1+n'+a+1} = \sum_{0 \leq j+k \leq n'} p_{jk}^{(n'+1,0)} s_{j+a-m,k+n'+a+1+m+1}
$$
Since
$$\col(s_{j+a-m, k+n'+a+m+2})_{0 \leq j+k \leq n'}$$
is a subvector of the column
$$Y^{n'+a+2} \in \cC_{M(n'+a+2)}$$
with rows indexed by
$$(X^{j+a-m} Y^{k+m})_{0 \leq j+k \leq n'},$$
from the definition of the extension $M(n'+a+2)$ in Step 1, we have
$$Y^{n'+a+2} = (y^{a+1} p^{(0,n'+1)})(X,Y)$$
and hence
$$\col(s_{j+a-m, k+n'+a+m+2})_{0 \leq j+k \leq n'} = \Phi \; \col(p_{jk}^{(0,n'+1)})_{0 \leq j+k \leq n'},$$
where $\Phi$ is the submatrix of $M(n'+a+2)$ with rows indexed by $(X^{j+a-m}Y^{k+m})_{0 \leq j+k \leq n'}$ and columns indexed by $(X^j Y^{k+a+1})_{0 \leq j+k \leq n'}$. It is readily checked that $\Phi = \Phi^*$. Thus, 
\begin{align*}
& \; \sum_{0 \leq j+k \leq n'} p_{jk}^{(n'+1,0)} s_{j+a-m, k+n'+a+m+2} \\
 =& \; \row(p_{jk}^{(0, n'+1)} )_{0 \leq j+k \leq n'} \, \Phi \, \col(p_{jk}^{(n'+1,0)} )_{0 \leq j+k \leq n'} \\
=& \; \row(p_{jk}^{(0,n'+1)})_{0 \leq j+k \leq n'} \, \Phi \, \col(p_{jk}^{(n'+1,0)})_{0 \leq j+k \leq n'} \\
=& \; \row(p_{jk}^{(0,n'+1)})_{0 \leq j+k \leq n'} \, \col(s_{j+n'+a+1-m, k+a+1+m})_{0 \leq j+k \leq n'},
\end{align*}
where \eqref{eq:Mar29e3}, with $m=0$, was used to obtain the last line above.

On the other hand, the number $s_{n'+a-m+1, m+n'+a+2}$ appears in the column $Y^{n'+a+2} \in \cC_{M(n'+a+2)}$ with row indexed by $X^{n'+a-m+1}Y^m$. Thus, \eqref{eq:Mar29e3} with $m = 0$ implies that
$$s_{n'+a-m+1,m+n'+a+2} = \sum_{0 \leq j+k \leq n'} p_{jk}^{(0,n'+1)} s_{j+n'-a-m+1,k+m+a+1}.$$
Thus, \eqref{eq:Mar29p1} holds when $\alpha = 0$ and $\beta = n'+a+1$.

\bigskip

\noindent {\bf Step 6:} {\it Check that \eqref{eq:Mar29p1} holds when $(\alpha,\beta) \in \Lambda^2_{n'+a+2} \, \setminus \, \{ (0,n'+a+2)\}$.}

\bigskip

Since 
$$\col(s_{\alpha+n'+a-m+1, \beta+m+1} )_{(\alpha,\beta) \in \Lambda^2_{n'+a+2} \setminus \{ (0,n'+a+2) \} }$$
is a subvector of
$$X^{n'+a+2-m}Y^m \in \cC_{M(n'+a+2)}$$
with rows indexed by
$$(X^{\tilde{\alpha}}Y^{\tilde{\beta}})_{(\tilde{\alpha},\tilde{\beta}) \in \Lambda^2_{n'+a+2} \setminus \{ (n'+a+2,0) \}},
$$
Steps 2-5 imply that
\begin{align}
\; & \col(s_{\alpha+n'+a-m+1, \beta+m+1} )_{ (\alpha, \beta) \in \Lambda^2_{n'+a+2} \setminus \{ (0,n'+a+2) \} } \nonumber \\
=& \; \col (\sum_{0 \leq j+k \leq n'} p_{jk}^{(n'+1, 0)} s_{\tilde{\alpha}+j+a-m+1,\tilde{\beta}+k+m} )_{(\tilde{\alpha},\tilde{\beta}) \in \Lambda^2_{n'+a+2} \setminus \{ (n'+a+2, 0) \} }. \label{eq:Apr10ee1}
\end{align}

If $(\alpha,\beta) \in \Lambda^2_{n'+a+2} \, \setminus \, \{(0,n'+a+2) \}$, then
$$\alpha = n' +a+2-u \quad {\rm and} \quad \beta=u \quad {\rm for} \quad u=0,\ldots, n'+a+1.$$
Similarly, if $(\tilde{\alpha},\tilde{\beta}) \in \Lambda^2_{n'+a+2} \, \setminus \, \{(0,n'+a+2) \}$, then
$$
\tilde{\alpha} = n'+a+1-\tilde{u} \quad {\rm and }\quad \tilde{\beta} =  \tilde{u} \quad {\rm for} \quad \tilde{u} = 1, \ldots, n'+a+2.
$$
Thus, \eqref{eq:Apr10ee1} can be rewritten in the following manner:
\begin{align*}
& \; \col( \sum_{0 \leq j+k \leq n'} p_{jk}^{(n'+1,0)} s_{\tilde{\alpha}+j+a-m+1,\tilde{\beta}+k+m})_{(\tilde{\alpha},\tilde{\beta}) \in \Lambda^2_{n'+a+1} \setminus \{ (n'+a+2, 0) \} } \\
=& \; \col( \sum_{0 \leq j+k \leq n'} p_{jk}^{(n'+1,0)} s_{\tilde{u}+j+n'+a-m+1,\tilde{u}+k+m+1}  )_{\tilde{u}=1}^{n'+a+2} \\
=& \; \col( \sum_{0 \leq j+k \leq n'} p_{jk}^{(n'+1,0)} s_{j+a-(\tilde{u}-1)-m, k+(\tilde{u}-1)+m+1} )_{\tilde{u}=1}^{n'+a+2} \\
=& \; \col(\sum_{0 \leq j+k \leq n'} p_{jk}^{(n+1,0)} s_{j+n'+a+1-u-m, k+u+m+1} )_{u=0}^{n'+a+1} \\
=& \; \col( \sum_{0 \leq j+k \leq n'} p_{jk}^{(n'+1,0)} s_{\alpha+j-u-m, \beta+k+m+1} )_{(\alpha,\beta) \in \Lambda^2_{n'+a+1} \setminus \{ (0,n'+a+2) \} },
\end{align*}
i.e., \eqref{eq:Mar29p1} holds when $(\alpha,\beta) \in \Lambda^2_{n'+a+2} \, \setminus \, \{ (0,n'+a+2)\}$.

\bigskip

\noindent {\bf Step 7:} {\it Check that \eqref{eq:Mar29p1} holds when $\alpha =0$ and  $\beta = n'+a+2$.}

\bigskip

We must show that 
$$
s_{n'+a-m+1, m+1+n'+a+2} = \sum_{0 \leq j+k \leq n'} p_{jk}^{(n'+1,0)} s_{j+a-m,k+n'+a+2+m+1}
$$
Since

$$\col(s_{j+a-m, k+n'+a+m+2})_{0 \leq j+k \leq n'}$$

is a subvector of the column

$$Y^{n'+a+2} \in \cC_{M(n'+a+2)}$$
with rows indexed by
$$(X^{j+a-m} Y^{k+m})_{0 \leq j+k \leq n'},$$
from the definition of the extension $M(n'+a+2)$ in Step 1, we have
$$Y^{n'+a+2} = (y^{a+1} p^{(0,n'+1)})(X,Y)$$
and hence
$$\col(s_{j+a-m, k+n'+a+m+2})_{0 \leq j+k \leq n'} = \wt{\Phi} \; \col(p_{jk}^{(0,n+1)})_{0 \leq j+k \leq n'},$$
where $\wt{\Phi}$ is the submatrix of $M(n'+a+2)$ with rows indexed by $(X^{j+a-m}Y^{k+m+1})_{0 \leq j+k \leq n'}$ and columns indexed by $(X^j Y^{k+a+1})_{0 \leq j+k \leq n'}$. It is readily checked that $\wt{\Phi} = \wt{\Phi}^*$. Thus, 
\begin{align*}
& \; \sum_{0 \leq j+k \leq n'} p_{jk}^{(n'+1,0)} s_{j+a-m, k+n'+a+2+m+1} \\
 =& \; \row(p_{jk}^{(0, n'+1)} )_{0 \leq j+k \leq n'} \, \wt{\Phi} \, \col(p_{jk}^{(n'+1,0)} )_{0 \leq j+k \leq n'} \\
=& \; \row(p_{jk}^{(0,n'+1)})_{0 \leq j+k \leq n'} \, \wt{\Phi} \, \col(p_{jk}^{(n'+1,0)})_{0 \leq j+k \leq n'} \\
=& \; \row(p_{jk}^{(0,n'+1)})_{0 \leq j+k \leq n'} \, \col(s_{j+n'+a+1-m, k+a+2+m})_{0 \leq j+k \leq n'},
\end{align*}
where \eqref{eq:Mar29e3}, with $m=0$, was used to obtain the last line above.

On the other hand, the number $s_{n'+a-m+1, m+n'+a+2}$ appears in the column $Y^{n'+a+2} \in \cC_{M(n'+a+2)}$ with row indexed by $X^{n'+a-m+1}Y^m$. Thus, \eqref{eq:Mar29e3} with $m = 0$ implies that
$$s_{n'+a-m+1,m+n+a+2} = \sum_{0 \leq j+k \leq n'} p_{jk}^{(0,n'+1)} s_{j+n'-a-m+1,k+m+1+a+1}.$$
Thus, \eqref{eq:Mar29p1} holds when $\alpha = 0$ and $\beta = n'+a+1$.

\bigskip

\noindent {\bf Step 8:} {\it Check that \eqref{eq:Mar29e8} holds.}

\bigskip

Assertion \eqref{eq:Mar29e8} can be verified in much the same way as \eqref{eq:Mar29e7} by a suitable adjustment of Steps 2-7.

\end{proof}

\section{Integral representations for $(s_{\gamma})_{\stackrel{0 \leq |\gamma | \leq 2n}{\gamma \in \NN_0^d}}$}
\label{sec:CFana}
In this section, we will formulate and prove main result, i.e., we will show that any truncated multisequence $(s_{\gamma})_{0 \leq |\gamma| \leq 2n}$ that has a rank preserving extension which gives rise to diagonalisable shift operators has a quasi-complex representing measure. In particular, if the shift operators have real eigenvalues, or, equivalently, the associated polynomial is radical with real variety, then the representing measure can be chosen to be signed.

\begin{defn}[Variety of a $d$-Hankel matrix]
\label{def:Variety}
Let $s = (s_{\gamma} )_{\stackrel{0 \leq |\gamma| \leq 2n}{\gamma \in \NN_0^d}}$ be a given real-valued multisequence and $M(n)$ be as in Definition \ref{def:Feb13dH}.
For $p \in \CC^d[ {\bf x} ]$, let
$$\cZ(p) = \{ z \in \CC^d: p(z) = 0 \}.$$
We will denote the {\it variety of $M(n)$} by $\cV(M(n))$ which is given by
\begin{equation}
\label{eq:Variety}
\cV(M(n)) := \bigcap_{\stackrel{p \in \RR^d_n[ {\bf x} ] }{p(X) = {\bf 0}} } \cZ(p).
\end{equation}
\end{defn}

\begin{rem}
\label{rem:Variety}
The variety of $M(n)$ appeared first in \cite{CF_memoir} (albeit implicitly in \cite{CF_memoir} and explicitly in subsequent works of Curto and Fialkow). It is also important to note that in Curto and Fialkow's papers on truncated multidimensional moment problems on $\RR^d$, if $M(n) \succeq 0$ and has a flat extension, i.e., a rank preserving extension $M(n+1)$ such that $M(n+1) \succeq 0$, then $\cV(M(n)) \subseteq \RR^d$.
\end{rem}

\begin{lem}
\label{lem:June21variety}
Let $s = (s_{\gamma} )_{0 \leq |\gamma| \leq 2n}$ be a given truncated multisequence of real numbers. If $M(n)$ has a rank preserving extensions $M(n+1)$ and $M(n+2)$, then $\cV(M(n+1)) = \cV(M(n+2))$. 
\end{lem}

\begin{proof}
Since $\rank M(n+2) = \rank M(n+1)$, Lemma \ref{lem:Mar31e3} implies that we $M(n+2)$ can be written as
$$M(n+2) = \begin{pmatrix} A & AW \\ W^T A & W^T A W \end{pmatrix},$$
where $A = M(n+1)$ and $W \in \RR^{ \ell(n+1,d) \times \ell(n+2,d)}$. Consequently, if $p(X) = {\bf 0} \in \cC_{M(n+1)}$ for some $p \in \RR^d_{n+1}[ {\bf x} ]$, $p(X) = {\bf 0} \in \cC_{M(n+2)}$. Indeed, if we write $p(x) = \sum_{0 \leq |\lambda| \leq n+1} p_{\lambda} x^{\lambda}$ and put $v = \col(p_{\lambda})_{0 \leq |\lambda | \leq n+1}$, then 
$$\tilde{v} = \begin{pmatrix} v \\ 0 \end{pmatrix}$$
will satisfy
$$M(n+2)\tilde{v} = \begin{pmatrix} A v \\ W^T A v \end{pmatrix} = {\bf 0}.$$ 
Thus, $\cV(M(n+2)) \subseteq \cV(M(n+1)).$

The reverse inclusion $\cV(M(n+1)) \subseteq \cV(M(n+2))$ can be proven by a straight forward adaptation of the proof of Theorem 2.4 in \cite{CF_several} to our present real multidimensional setting. 
\end{proof}

In the following theorem and corollary, corresponding to a monomial basis $\cB = \{ X^{\lambda^{(a)}} \}_{a=1}^r$ of $\cC_{M(n)}$, we shall let
$$\cB + {\rm e}_j = \{ X^{\lambda} + {\rm e}_j: X^{\lambda} \in \cB \}$$
and $M_{\cB}(n)$ and $M_{\cB, \cB+{\rm e}_j }(n+1)$ denote the principal submatrix of $M(n)$ with rows and columns indexed by $\cB$ and the submatrix of $M(n+1)$ with rows indexed by $\cB$ and columns indexed by $\cB+{\rm e}_j$, respectively.

\begin{thm}
\label{thm:Feb7m1}
Let $s = (s_{\gamma} )_{0 \leq |\gamma| \leq 2n}$ be a given multisequence of real numbers and $r := \rank M(n)$. $s$ has an $r$-atomic quasi-complex representing measure $\mu = \sum_{a=1}^r \varrho_a \delta_{z^{(a)}}$ if and only if $M(n)$ has a rank preserving extension $M(n+1)$ and
$$
\text{$T_{j}$ is diagonalisable for all $j=1,\ldots, d$,}
$$
where $T_j := M_{\cB}(n)^{-1} M_{\cB, \cB+ {\rm e}_j}(n+1)$ for $j=1,\ldots, d$. 
In this case,
\begin{equation}
\label{eq:m1e2}
\supp \mu = \cV(M(n+1))
\end{equation}
and the weights $\varrho_1, \ldots, \varrho_r$ can be computed via \eqref{eq:WEIGHTS}. If, in addition,
$$\cV(M(n+1)) \subseteq \RR^d,$$ 
then $\mu = \mu_+ - \mu_-$ is a $\rank M(n)$-atomic signed representing measure for $s$ and the measures $\mu_{\pm}$ in the above Jordan decomposition of $\mu$ satisfy
\begin{equation}
\label{eq:INERTPOS}
\card \supp \mu_{\pm} = i_{\pm}(M(n)).
\end{equation}
\end{thm}

\begin{proof}
Suppose $\mu = \sum_{a=1}^r \varrho_a \delta_{z^{(a)}}$ is an $r$-atomic representing measure for $s$ and consider the real multisequence $(s_{\gamma})_{\gamma \in \NN_0^d}$, where
$$s_{\gamma} = \int_{\CC^d} z^{\gamma} d\mu(z) \quad {\rm for} \quad \gamma \in \NN_0^d.$$
Lemma \ref{lem:Feb22m1} implies that
$$\rank M(n+1) \leq \card \supp \mu.$$
Since $r := \rank M(n) = \card \supp \mu$, 
$$\rank M(n+1) \leq r.$$
The equality $\rank M(n+1) = \rank M(n)$ then follows from 
$$\rank M(n) \leq \rank M(n+1).$$
We may continue continue in this manner and verify that $\rank M(\infty) = \rank M(n)$. Thus, the basis $\cB = \{ X^{\lambda^{(1)}}, \ldots, X^{\lambda^{(r)}} \}$ for $M(n)$ can easily be identified with a basis for $\cC_{M(\infty)}$ (by just interpreting the column vectors appropriately). With a slight abuse of notation we will use $\cB$ for both bases. Consequently,
$$M_{\cB}(n) = M_{\cB}(\infty)$$
and
$$M_{\cB,\cB+{\rm e}_j}(n+1) = M_{\cB, \cB+{\rm e}_j}(\infty) \quad\quad {\rm for} \quad j=1,\ldots,d$$
and hence the factorisations \eqref{eq:FF1} and \eqref{eq:FF2} and the invertibility of $V_{\cB}$ imply that
\begin{equation}
\label{eq:DIAG}
T_j = M_{\cB}(n)^{-1} M_{\cB, \cB + {\rm e}_j}(n+1) = V_{\cB}^{-1} Z^{(j)} V_{\cB} \quad\quad {\rm for }\quad j=1,\ldots,d
\end{equation}
i.e., $T_j$ is diagonalisable for $j=1,\ldots,d$. 

Conversely, suppose $M(n)$ has a rank preserving extension $M(n+1)$. Theorem \ref{thm:June14k1} asserts the existence of rank preserving extensions $(M(n+m))_{m=2}^{\infty}$ such that $\rank M(\infty) = \rank M(n)$. Let $\cI$ and $\cA :=\RR[x_1, \ldots, x_d] / \cI $ be the polynomial ideal and quotient algebra, respectively, associated with $M(\infty)$ (see Section \ref{sec:full}). Let $\cB = \{ X^{\lambda^{(a)}} \}_{a=1}^r$ be a monomial basis for $\cC_{M(n)}$, where $r = \rank M(n) = \rank M(\infty)$. Since $\rank M(n) = \rank M(\infty)$, $\cB$ (when suitably embedded in $\cC_{M(\infty)}$) is also a basis for $\cC_{M(\infty)}$. Thus, as 
$$T_j = M_{\cB}(n)^{-1} M_{\cB,\cB+ {\rm e}_j}(n+1)$$
is nothing but a matrix representation for $\Theta_j: \cA \to \cA$ with respect to the basis $\cB$ (see \eqref{eq:Feb15shift}), we have that $\Theta_j$ is diagonalisable for $j=1,\ldots,d$. We may use Theorem \ref{thm:Oct28fas} to obtain an $r$-atomic quasicomplex representation measure $\mu = \sum_{a=1}^r \varrho_a \delta_{z^{(a)}}$ for $s$ such that
$$\supp \mu = \cV(\cI)$$
and the weights $\varrho_1, \ldots, \varrho_r$ may be computed by \eqref{eq:WEIGHTS}. If we apply Lemma \ref{lem:June21variety} to $M(n+1), M(n+2), \ldots$, then we obtain 
$$\cV(M(n+1)) = \cV(\cI)$$
and hence \eqref{eq:m1e2} holds.

%


Finally, if $\cV(M(n+1)) \subseteq \RR^d$, then $\supp \mu \subseteq \RR^d$ follows from \eqref{eq:Feb18v1}. Assertion \eqref{eq:INERTPOS} follows from the fact that $M_{\cB}(n) = M_{\cB}(\infty)$ and \eqref{eq:Hahn1}.
\end{proof}

The following corollary is a restatement of the final conclusion of Theorem \ref{thm:Feb7m1}.

\begin{cor}
\label{cor:Nov1}
Let $s = (s_{\gamma} )_{0 \leq |\gamma| \leq 2n}$ be a given multisequence of real numbers and $\kappa_{\pm} := i_{\pm}(M(n))$. $s$ has a representing measure $\mu$ with Jordan decomposition $\mu = \mu_+ - \mu_-$ with $\card \supp \mu_{\pm} = \kappa_{\pm}$ if and only if $M(n)$ has a rank preserving extension $M(n+1)$ with the following properties{\rm :}
$$
\text{$T_{j}$ is diagonalisable for all $j=1,\ldots, d$,}
$$
where $T_j := M_{\cB}(n)^{-1} M_{\cB, \cB+ {\rm e}_j}(n+1)$ for $j=1,\ldots, d$ and 
$$\cV(M(n+1)) \subseteq \RR^d.$$ In this case $\mu$ can be constructed as in Theorem {\rm \ref{thm:Feb7m1}}.
\end{cor}

\begin{rem}
\label{rem:POSCASE}
If $s$ in Theorem \ref{thm:Feb7m1} is positive definite, i.e., $M(n) \succeq 0$, then Lemma \ref{lem:Mar31e3} implies that the rank preserving extension $M(n+1)$ must satisfy $M(n+1) \succeq 0$. Thus, in the terminology of Curto and Fialkow (see, e.g., \cite{CF_memoir}, \cite{CF_several}), $M(n+1)$ is a {\it flat extension} of $M(n)$. Repeated applications of Lemma \ref{lem:Mar31e3} serve to show that all of the subsequent extensions $(M(n+a))_{a=2}^{\infty}$ would also flat and the Pontryagin space constructed in Section \ref{sec:Pontryagin} is a Hilbert space. In this case, the diagonisability of $T_j := M_{\cB}(n)^{-1} M_{\cB, \cB + {\rm e}_j }(n+1)$ for $j=1,\ldots, d$, or, equivalently, that $\cI$ is a real radical ideal, comes as a direct consequence of $M(n) \succeq 0$ (see Lemma 5.2 in \cite{Laurent}). 

\end{rem}

\begin{thm}
\label{thm:Oct11e1}
Let $s = (s_{\gamma})_{0 \leq |\gamma| \leq 2n'}$ be a given real-valued truncated multisequence. If $M(n')$ has an extension $M(n'+1)$ such that
\begin{equation}
\label{eq:poly1}
p^{(n'+1){\rm e}_j}(X) = {\bf 0 } \in \cC_{M(n'+1)} \quad \quad {\rm for} \quad j=1,\ldots, d,
\end{equation}
form some polynomials $p^{(n'+1){\rm e}_j}(x) \in \RR^d_{n'+1}[{\bf x}]$ which have $n'+1$ distinct real zeros, i.e.,
$$p^{(n'+1){\rm e}_j}(x_1, \ldots, x_d) = \prod_{a=1}^{n'+1} (x_j - w_j^{(a)} ) \quad \quad {\rm for} \quad j=1,\ldots,
$$
where $w_j^{(1)}, \ldots, w_j^{(n'+1)}$ are distinct for any fixed $j \in \{1, \ldots d \}$, then  $s$ has a signed representing measure $\mu = \mu_+ - \mu_-$ with the following properties{\bf :}
\begin{equation}
\label{eq:M3}
\rank M(n') \leq \card \mu \leq (n'+1)^d
\end{equation}
and
\begin{equation}
\label{eq:M4}
\card \supp \mu_{\pm} \geq i_{\pm}(M(n')).
\end{equation}
\end{thm}

\begin{proof}

We will use Lemma \ref{lem:Mar28l1} to obtain extensions $(M(n'+k))_{k=2}^{\infty}$ (determined by \eqref{eq:poly1}) such that 
$$\rank M(\infty) = \rank M(n'd).$$
Moreover, since the extension $M(n'+2)$ has column relations
$$
X^{(n'+1){\rm e}_j+{\rm e}_{j'}} = - \sum_{0 \leq |\lambda| \leq n'} p_{\lambda}^{(n'+1){\rm e}_j} X^{\lambda+{\rm e}_{j'}} \quad \quad {\rm for} \quad j,j'=1,\ldots,d,
$$
Lemma \ref{lem:July19k1} implies that $p^{(n'+1){\rm e}_j}(X) = {\bf 0} \in \cC_{M(n'+2)}$ for $j=1,\ldots, d$. One can inductively continue in this manner to obtain
\begin{equation}
\label{eq:INDUCTIVE}
p^{(n'+1){\rm e}_j}(X) = {\bf 0} \in \cC_{M(\infty)} \quad\quad {\rm for} \quad j=1,\ldots, d.
\end{equation}
In addition, we may use Lemma \ref{lem:June21variety} to obtain
\begin{equation}
\label{eq:VARIETYCHAIN}
\cV(M(n'd+1)) = \cV(M(\infty)).
\end{equation}

Let $\cI$, $\cA :=\RR[x_1, \ldots, x_d] / \cI$ and $\Theta_j: \cA \to \cA$ with $j=1,\ldots,d$ be the associated ideal, quotient space and shift operators associated with $(s_{\gamma})_{\gamma \in \NN_0^d}$, respectively, as in Section \ref{sec:full}. In view of \eqref{eq:INDUCTIVE},
$$\cV(\cI) \subseteq \{ w^{(1)}, \ldots, w^{((n'+1)^d)} \},$$
where 
$$w^{(a)} := (w_1^{(a)}, \ldots, w_d^{(a)}) \quad\quad {\rm for} \quad a = 1, \ldots, (n'+1)^d.$$
If $r := \rank M(\infty) = \rank M(n'd)$, then, in view of \eqref{eq:Feb15d4}, we have
$$\card \cV(\cI) \leq r.$$
We endeavour to show that the shift operators $\Theta_j: \cA \to \cA$ for $j=1,\ldots, d$ are diagonalisable. We claim that 
\begin{equation}
\label{eq:Radical}
\card \cV(\cI) = r,
\end{equation}
or, equivalently (see, e.g., Theorem 2.6 in \cite{Laurent}), $\cI$ is {\it radical}. The fact that $\cI$ is radical, and hence real radical, follows directly from Seidenberg's Lemma (see, e.g., Proposition 3.7.15 in \cite{KR}). Indeed, if we let 
$$\psi^{(n'+1){\rm e}_j}(x_j) = p^{(n'+1){\rm e}_j}(x) \quad\quad {\rm for}  \quad j=1, \ldots, d,$$
then 
$$\text{$\{ \psi^{(n'+1){\rm e}_j}(x_j) \}_{j=1}^d$ are square free univariate polynomials belonging to $\cI$}$$
and thus we have
$$\gcd(\psi, \psi') = 1,$$
where by $\gcd(\psi, \psi')$ we mean the monic greatest common divisor of $\psi$ and by $ \psi'$ we mean the formal derivative of $\psi$. Consequently, Seidenberg's Lemma (see, e.g., Proposition 3.7.15 in \cite{KR}) asserts that $\cI$ is radical. Thus, \eqref{eq:Radical} is in force and suppose, without loss of generality,
$$\cV(\cI) = \{ w^{(1)}, \ldots, w^{(r)} \} \subseteq \{ w^{(1)}, \ldots, w^{((n'+1)^d)} \}.$$
In view of assertion (ii) in Theorem \ref{thm:June20t1}, we have
$$\sigma(\Theta_j) = \{ \pi_j(w^{(1)}), \ldots, \pi_j(w^{(r)}) \} \quad \quad {\rm for} \quad j=1,\ldots, d$$
and since $\pi_j(w^{(1)}), \ldots, \pi_j(w^{((n'+1)^d)})$ are distinct for any $j \in \{ 1, \ldots, d \}$, we have that $\Theta_j$ is diagonalisable for $j=1,\ldots, d$. Thus, we may use Theorem \ref{thm:Oct28fas} to obtain an $r$-atomic signed representing measure $\mu = \sum_{a=1}^r \varrho_a \delta_{w^{(a)}}$ for $(s_{\gamma})_{0 \leq |\gamma| \leq n'd}$, and, in particular for $s$.

Finally, the lower bound in \eqref{eq:M3} was follows immediately from Lemma \ref{lem:Feb22supp}, while the upper bound in \eqref{eq:M3} follows immediately from \eqref{eq:Radical}, 
$$\cV(\cI) \subseteq \{ w^{(1)}, \ldots, w^{((n'+1)^d)} \}$$ 
and 
$$\card \{ w^{(1)}, \ldots, w^{((n'+1)^d)} \} = (n'+1)^d.$$
The lower bounds for $\card \supp \mu_{\pm}$ follow immediately from \eqref{eq:INERTPOS} and the Interlacing Theorem (see, e.g., Theorem 4.3.15 in \cite{HJ}). Indeed if we order the eigenvalues of the Hermitian matrices $M(n')$ and $M(n'd+1)$ by 
$$\lambda_1(M(n')) \leq \ldots \leq \lambda_{\ell(n',d)}(M(n'))$$
and
$$\lambda_1(M(n'd+1)) \leq \ldots \leq \lambda_{\ell(n'd+1,d)}(M(n'd+1)),$$
respectively, then the Interlacing Theorem asserts that
$$\lambda_k(M(n'd+1)) \leq \lambda_k(M(n)) \leq  \lambda_{k+\ell(n'd+1,d)-\ell(n',d)}(M(n'd+1)) \quad\quad {\rm for} \quad k=1,\ldots, \ell(n,d)$$
and hence
$$\card \supp \mu_- = i_-(M(n'd+1)) \geq i_-(M(n'))$$
and
$$i_+(M(n')) \leq i_+(M(n'd+1)) = \card \supp \mu_+.$$

\end{proof}

\section{Examples}
\label{sec:Ex}
In this section we will showcase a number of examples which utilise Theorem \ref{thm:Feb7m1} and Theorem \ref{thm:Oct11e1}. In particular, we will compute a signed representing measure for a positive definite truncated bisequence of degree $6$ which is known not to have a positive representing measure (\cite[Example 3.11]{Fialkow_FV}).

\begin{example}
\label{ex:ONEVAR}
The univariate truncated sequence $s = (s_{\gamma})_{\gamma=0}^6$ given by
$$M(3) = \begin{pmatrix} 0 & 0 & 0 & 1 \\ 0 & 0 & 1 & 0 \\
0 & 1 & 0 & -2 \\
1 & 0 & -2 & 0 \end{pmatrix}$$
does not have a $4$-atomic signed representing measure. Indeed, Lemma \ref{lem:Mar31e3} implies that every rank preserving extension $M(4)$ must be of the form
$$M(4) = \begin{pmatrix} 0 & 0 & 0 & 1 & 0 \\ 0 & 0 & 1 & 0 & -2 \\
0 & 1 & 0 & -2 & 0 \\ 
1 & 0 & -2 & 0 & s_9 \\ 
0 & -2 & 0 & s_9 & s_{10} 
\end{pmatrix},$$
where 
\begin{align*}
s_9 =& \; \alpha_0 +4 \quad\quad {\rm for} \quad \alpha_0 \in \RR \\
\tag*{and} \\
s_{10} =& \; 0.
\end{align*} 
Moreover, $M(4)$ only has the column relation
$$p(X) = {\bf 0 } \in \RR^5 \quad\quad {\rm for} \quad p(x) = x^4+2x^2 - \alpha_0 \quad {\rm with } \quad \alpha_0 \in \RR.$$
It is ready checked that the zero set of $p(x)$ is given by
$$\mathcal{Z}(p) = \{ \pm \sqrt{-1 + \tilde{\alpha}_0 }, \pm \sqrt{ -1 - \tilde{\alpha}_0 } \} \quad\quad {\rm for} \quad \tilde{\alpha}_0 := \sqrt{\alpha_0+1}.$$
Thus, for any choice of $\alpha_0 \in \RR$, $p(x)$ has at least two distinct nonreal complex zeros. Thus, Corollary \ref{cor:Nov1} implies that $s$ cannot have a $(2,2)$ representing measure.

Theorem \ref{thm:Feb7m1} asserts corresponding to any choice of $\alpha_0 \in \RR$ such that $\card \mathcal{Z}(p) =4$ (and hence the shift matrix $T_1$ is diagonalisable), there exists of a $4$-atomic quasicomplex measure $\mu = \sum_{a=1}^4 \varrho_a \delta_{z_a}$ for $s$ with $\card \supp \mu \cap (\CC \setminus \RR) \geq 2$. For instance, if $\alpha_0 = 3$, then
$$\mathcal{Z}(p) = \{ 1, -1, i\sqrt{3}, -i\sqrt{3} \}$$
and if we let $z_1 = 1, z_2 = -1, z_3 i \sqrt{3}$ and $z_4 = - i \sqrt{3}$, then
\begin{align*}
\begin{pmatrix} \varrho_1 \\ \varrho_2 \\ \varrho_3 \\ \varrho_4 \end{pmatrix}
 =& \; \begin{pmatrix} 1 & z_1 & z_1^2 & z_1^3 \\
1 & z_2 & z_2^2 & z_2^3 \\
1 & z_3 & z_3^2 & z_3^3 \\
1 & z_4 & z_4^2 & z_4^3
\end{pmatrix}^{-T} 
M(4) \begin{pmatrix} 1 \\ 0 \\ 0 \\ 0 \end{pmatrix} \\
=& \; \begin{pmatrix} \frac{1}{8} \\[6pt] -\frac{1}{8} \\[6pt] \frac{i \sqrt{3}}{24} \\[6pt] -\frac{i \sqrt{3}}{24} \end{pmatrix}.
\end{align*}

\end{example}

We will now use Theorem \ref{thm:Feb7m1} to compute an $r$-atomic quasicomplex representing measure for a truncated multisequence which is positive definite and is known {\it not to have a positive representing measure} \cite[Example 3.11]{Fialkow_FV}. 

\bigskip

\noindent {\bf Warning:} For the convenience of the reader, we will use $x, y$ and $X, Y$ in place of $x_1, x_2$ and $X_1, X_2$, respectively, in the following example.

\bigskip

\begin{example}
\label{ex:QC}
Let $s = (s_{\gamma_1,\gamma_2})_{0 \leq \gamma_1+\gamma_2 \leq 6}$ be the real-valued truncated multisequence given by
$$M(3) = \begin{pmatrix} M(2) & B \\ B^T & C \end{pmatrix} \succeq 0,$$
where
$$
M(2) :=  \begin{pmatrix} 14 & \frac{7}{2} & -\frac{67}{8} & \frac{79}{4} & \frac{1055}{16} & \frac{18195}{64} \\[6pt]
\frac{7}{2} & \frac{79}{4} & \frac{1055}{16} & -\frac{67}{8} & -\frac{1935}{32} & -\frac{-43115}{128}  \\[6pt]
-\frac{67}{8} & \frac{1055}{16} & \frac{18195}{64} & -\frac{1935}{32} & -\frac{43115}{128} & -\frac{926695}{512}  \\[6pt]
\frac{79}{4} & -\frac{67}{8} & -\frac{1935}{32} & \frac{1055}{16} & \frac{18195}{64} & \frac{336151}{256} \\[6pt]
\frac{1055}{16} & -\frac{1935}{16} & - \frac{43115}{128} & \frac{18195}{64} & \frac{336151}{256} & \frac{6407195}{1024} \\[6pt]
\frac{18195}{64} & -\frac{43115}{128} & -\frac{926695}{512} & \frac{336151}{256} & \frac{6407195}{1024} & \frac{124731423}{4096}
\end{pmatrix}, 
$$
$$B := \begin{pmatrix} 
-\frac{67}{8} & -\frac{1935}{32} & -\frac{43115}{128} & -\frac{926695}{512} \\[6pt]
\frac{1055}{16} & \frac{18195}{64} & \frac{336151}{256} & \frac{6407195}{1024} \\[6pt]
\frac{18195}{64} & \frac{336151}{256} & \frac{6407195}{1024} & \frac{124731423}{4096} \\[6pt]
-\frac{1935}{32} & -\frac{43115}{128} & -\frac{926695}{512} & -\frac{19736547}{2048} \\[6pt]
-\frac{43115}{128} & -\frac{926695}{512} & -\frac{19736547}{2048} & -\frac{419176415}{8192} \\[6pt] 
-\frac{926695}{512} & -\frac{19736547}{2048} & -\frac{419176415}{8192} & -\frac{8894873563}{32768}
\end{pmatrix}
$$
and
$$C:= \begin{pmatrix} \frac{18195}{64} & \frac{336151}{256} & \frac{6407195}{1024} & \frac{124731423}{4096} \\[6pt] 
\frac{336151}{256} & \frac{6407195}{1024} & \frac{124731423}{4096} & \frac{2469281827}{16384} \\[6pt] 
\frac{6407195}{1024} & \frac{124731423}{4096} & \frac{2469281827}{16384} & \frac{49568350247}{65536} \\[6pt]
\frac{124731423}{4096}  & \frac{2469281827}{16384} & \frac{49568350247}{65536} & \frac{1006568996907}{262144}
\end{pmatrix}.
$$
A basis for $\cC_{M(3)}$ is given by 
$$\cB = \{ 1, X, Y, X^2, XY, Y^2, X^2 Y, XY^2 \}$$
and $M(3)$ has the column relations
\begin{equation}
\label{eq:CC1}
X^3 =Y \in \cC_{M(3)}
\end{equation}
and
\begin{equation}
\label{eq:CC2}
Y^3 = 3 X + \frac{45}{4} Y - 13 X^2 + \frac{65}{4} X Y - \frac{13}{4} Y^2 - 22 X^2 Y +\frac{35}{4} XY^2 \in \cC_{M(3)}.
\end{equation}
As observed by Fialkow in \cite{Fialkow_FV}[Example 4.18], since $M(3)$ is {\it recursively determinate} (see \cite{Fialkow_FV}), any possible positive semidefinite extension $M(4)$ must be given by
\begin{equation}
\label{eq:X4}
X^4 = XY
\end{equation}
and
\begin{equation}
\label{eq:Y4}
Y^4 = 3XY + \frac{45}{4}Y^2 - 13X^2 Y + \frac{65}{4} XY^2 - \frac{13}{4} Y^3 - 22X^2 Y^2 + \frac{35}{4}XY^3.
\end{equation}
However, the extension $M(4)$ guided by \eqref{eq:X4} and \eqref{eq:Y4} is not positive semidefinite and hence cannot have a {\it positive representing measure}.

We will now use Theorem \ref{thm:Feb7m1} to compute a $12$ atomic quasicomplex representing measure for $s$. Let
$$
p^{(3,0)}(x,y) = x^3 + \sum_{0 \leq a+b \leq 2} p^{(3,0)}_{ab} \, x^a y^b = x^3-y
$$
and
\begin{align*}
p^{(0,4)}(x,y) =& \; y^4 + \sum_{0 \leq a+b \leq 3} p^{(0,4)}_{ab}  \\
=& \; y^4+\left(\frac{47163262297569}{51374768128} \right)x-\left(\frac{3691292821051}{3022045184} \right) y \\
-& \; \left(\frac{65357522194769}{51374768128}\right)x^2+\left(\frac{10195141857683}{6421846016}\right)xy  \\
-& \; \left(\frac{76486966005}{200682688}\right)y^2+x^3+\left(\frac{169021630829}{755511296}\right)x^2y+10xy^2-y^3
\end{align*}
Then one can check that it is possible to consistently define an extension $M(4)$ via the column relations
$$X^4 = -\sum_{0 \leq a+b \leq 3} p_{ab}^{(4,0)} X^{a+1} Y^b$$
and
$$Y^4 = -\sum_{0 \leq a+b \leq 3} p_{ab}^{(0,4)} X^a Y^{b+1},$$
in which case $\rank M(4) = 11$, since the only four column relations are
\begin{align*}
X^3 =& \; Y \\
X^4 =& \; XY \\
XY^3 =& \; Y^2 \\
Y^4 =& \; -\sum_{0 \leq a+b \leq 3} p^{(0,3)}_{ab} X^a Y^{b+1}.
\end{align*}

Lemma \ref{lem:Mar28l1} with $n' = 2$ implies that $M(4)$ has successive extensions $(M(3+a))_{a=2}^{\infty}$ such that
$\rank M(\infty) = \rank M(6) = \rank M(5).$
Let $\cI$, $\cA := \RR[x_1, \ldots, x_d] / \cI$ and $\Theta_j: \cA \to \cA$ with $j=1,\ldots,d$ be the associated ideal, quotient space and shift operators associated with $(s_{\gamma})_{\gamma \in \NN_0^d}$, respectively, as in Section \ref{sec:full}. One can check that $\rank M(5) = 12$. Moreover,
$$\cI = \langle p^{(4,0)}(x,y), p^{(0,4)}(x,y) \rangle$$
satisfies $\card \cV(\cI) \leq \rank M(\infty) = 12$ and a basis for $\cC_{M(5)}$ is given by
$$
\cB = \{ X^a Y^b \}_{0 \leq a + b \leq 2} \bigcup \{  X^2Y,  XY^2, Y^3, X^2 Y^2, X^2 Y^3, XY^4 \}. 
$$
Since $\cI \cap \RR[{\bf x}] = \langle \varphi(x) \rangle$ and $\cI \cap \RR[ { \bf y} ] = \langle \psi(y) \rangle$, where

\begin{align*}
\varphi(x) :=& \; 51374768128x^{12}-51374768128x^9+513747681280x^7-19580663297280x^6 \\
&+ \; 11493470896372x^5+81561134861464x^4-62700603189739x^3 \\ &- \; 65357522194769x^2+47163262297569x
\end{align*}
and
\begin{align*}
\psi(y) :=& \; 135596857365312460266673014833152y^{12} \\
&- \;406790572095937380800019044499456y^{11} \\
&- \;15463486326449173016863163342127104y^{10} \\ &- \;186521952295893441652500941095567360y^9 \\
&+ \; 59019457490413813459703155070657888256y^8 \\
&+ \; 180596151418034206029765192097508884480y^7 \\
&- \; 5991020954875276874502839619015040745472y^6\\
&- \; 6974988944845260455200342338544544855744y^5\\
&+ \; 185832155803676288347648245170522695494448y^4\\
&- \;236371187104770223823228997175985270647387y^3\\
&- \; 314733260415193681665802613874461583211034y^2\\
&+ \; 104908701893148235312716198616535564271009y
\end{align*}
One can check that $\gcd(\varphi, \varphi') = \gcd(\psi, \psi') = 1$ and hence $\varphi$ and $\psi$ are squarefree. Thus, we can use, e.g., Corollary 3.7.16 in \cite{KR} to obtain
$$\sqrt{\cI} = \cI + \langle {\rm sqfree}(\varphi), {\rm sqfree}(\psi) \rangle = \cI +   \langle \varphi, \psi \rangle = \cI$$
and hence $\cI$ is radical. Consequently, the shift operators $\Theta_j: \cA \to \cA$ are diagonalisable which is equivalent to the shift matrices
$$T_j =  M_{\cB}(5)^{-1} M_{\cB, \cB + {\rm e}_j}(6) \quad\quad {\rm for} \quad j=1,\ldots, d$$
being diagonalisable. Thus, Theorem \ref{thm:Feb7m1} asserts that $s$ has a $12$-atomic quasicomplex representing measure $\mu = \sum_{a=1}^{12} \varrho_a \delta_{(u_a, v_a)},$
where $$\{ (u_a, v_a) \}_{a=1}^{12} = \cV(\cI)$$
and $\varrho_1, \ldots, \varrho_{12}$ are given by \eqref{eq:WEIGHTS}.

Unfortunately, $\cV(\cI)$ seems too cumbersome to be computed in closed form. Nevertheless, using Maple, one can obtain the following $16$ digit approximations for $\{ (u_a, v_a) \}_{a=1}^{12},$ namely
\begin{align*}
(u_1, v_1) =& \;  (0,0) \\
(u_2, v_2) \approx & \; (0.6585688908335371, 0.2856298787956733) \\
(u_3, v_3) \approx& \; ( 1.367436555562332, 2.556946004386548) \\
(u_4, v_4) \approx& \; (1.698286037615674, 4.898154923194823) \\
(u_5, v_5) \approx& \; (1.995426735187136, 7.945246215103228) \\
(u_6, v_6) \approx& \; ( 1.161581775737723+2.514035960495372i, \\ & \; -20.45761353392138-5.713297928140102i) \\
(u_7, v_7) \approx& \; (-1.369713057599855+2.577649039498679i, \\ & \; 24.73251126202266-2.618718696770263i), \\
 (u_8,v_8) \approx& \; (-0.9983878163914217, -0.9951712423919440) \\
 (u_9, v_9) \approx& \; (-2.001787809774772, -8.021472900594320 )  \\
 (u_{10}, v_{10}) \approx& \; (-2.303280029308220, -12.21912833469656) \\
 (u_{11}, v_{11}) \approx& \; (-1.369713057599855-2.577649039498679i, \\ & \; 24.73251126202266+2.618718696770263i) \\
 (u_{12}, v_{12}) =& \; (1.161581775737723-2.514035960495372i, \\ & \; -20.45761353392138+5.713297928140102i)
\end{align*}
in which case we can use \eqref{eq:WEIGHTS} to obtain $16$-digit approximations for the weights $\varrho_1, \ldots, \varrho_{16}$ given by
\begin{align*}                                               
   \varrho_1  \approx& \; 0.75251705199900  \\                         
     \varrho_2  \approx& \;   8.157396366263352 \\                                                  
     \varrho_3  \approx& \;    1.23049608230004 \\
        \varrho_4  \approx& \;  -0.19997425644493\\
        \varrho_5  \approx& \;  1.049869493443501 \\
        \varrho_6  \approx& \;  -4.4036738951\cdot 10^{-7}+0.1181009400413 \cdot 10^{-5}i \\
       \varrho_7  \approx& \;   4.91682149450\cdot 10^{-7}+5.3431516336\cdot 10^{-8}i\\
        \varrho_8 \approx& \;   1.011062561963828 \\
        \varrho_9  \approx& \;    1.003988449908174\\
        \varrho_{10}  \approx& \;  0.9946441479402077 \\                                                 
       \varrho_{11}  \approx& \;   4.91682148481\cdot 10^{-7}-5.3431516336\cdot 10^{-8}i \\
       \varrho_{12}  \approx& \;   -4.4036739099\cdot 10^{-7}-0.1181009399657\cdot 10^{-5}i.
\end{align*}

\end{example}

\begin{example}[Example \ref{ex:QC} revisited]
\label{ex:SIGNED}
Let $s = (s_{\gamma_1,\gamma_2})_{0 \leq \gamma_1+\gamma_2 \leq 6}$ be the real-valued truncated multisequence defined in Example \ref{ex:QC}. We saw in Example \ref{ex:QC} has a $12$ atomic quasicomplex representing measure, where four of the atoms are nonreal complex numbers. We shall see that $s$ has a signed representing measure with $15$ atoms.

We will now use Theorem \ref{thm:Feb7m1} to compute a $15$ atomic signed representing measure for $s$. Let
\begin{align*} 
p^{(0,3)}(x,y) =& \; y^3 +\sum_{0 \leq a+b \leq 2} p^{(0,3)}_{ab} \, x^a y^b \\
=& \; y^3-3x-\left(\frac{45}{4}\right)y+13x^2-\left(\frac{65}{4}\right)xy+\left(\frac{13}{4}\right)y^2+22x^2y-\left(\frac{35}{4}\right)xy^2
\end{align*}
and
\begin{align*}
p^{(5,0)}(x,y) =& \; x^5 + \sum_{0 \leq a+b \leq 4} p^{(5,0)}_{ab}  \\
=& \; x^5+\left(\frac{4367271}{897715}\right)x-\left(\frac{12131209}{3590860}\right)x^2-\left(\frac{3792527}{718172}\right)x^3+x^4
\end{align*}
Then one can check that it is possible to consistently define an extension $M(4)$ via the column relation
$$Y^4 = -\sum_{0 \leq a+b \leq 3} p_{ab}^{(0,4)} X^a Y^{b+1}.$$
Next, one can check that it is possible to consistently define an extension $M(5)$ via the column relations
$$X^5 = -\sum_{0 \leq a+b \leq 3} p_{ab}^{(5,0)} X^{a} Y^b$$
and
$$Y^5 = -\sum_{0 \leq a+b \leq 3} p_{ab}^{(0,4)} X^a Y^{b+2},$$
in which case $\rank M(5) = 14$, with seven column relations given by
\begin{align}
Y^3 =& \; X, \label{eq:1} \\
XY^3 =& \; X^2, \label{eq:2} \\
Y^4 =& \; XY, \label{eq:3}\\
X^5 =& \; -\sum_{0 \leq a+b \leq 3} p_{ab}^{(5,0)} X^{a} Y^b, \label{eq:4}\\
X^2 Y^3 =& \; X^3, \label{eq:5}\\
XY^4 =& \; X^2 Y, \label{eq:6}\\ 
Y^5 =& \;  X Y^2. \label{eq:7}
\end{align}
The column relations \eqref{eq:1}-\eqref{eq:7} can be used to consistently define an extension $M(6)$ with $\rank M(6) = 15$, with 13 column relations by \eqref{eq:1}-\eqref{eq:7} and 
\begin{align*}
X^6 =& \;  -\sum_{0 \leq a+b \leq 3} p_{ab}^{(5,0)} X^{a+1} Y^b,  \\
X^5Y  =& \; -\sum_{0 \leq a+b \leq 3} p_{ab}^{(5,0)} X^{a} Y^{b+1},  \\
X^3 Y^3 =& \; X^4, \\
X^2 Y^4 =& \; X^3 Y, \\
X Y^5 =& \; X^2 Y^2, \\
Y^6 =& \; X Y^3. 
\end{align*}

Finally, the above column relations can be used to consistently define an extension $M(7)$ such that $\rank M(7) = 15$, i.e., $M(6)$ has a rank preserving extension. Theorem \ref{lem:Mar28l1} with $n' = 6$ implies that $M(7)$ has successive extensions $(M(6+a))_{a=2}^{\infty}$ such that
$\rank M(\infty) = \rank M(7) = \rank M(6) = 15.$
Let $\cI$, $\cA := \RR[x_1, \ldots, x_d] / \cI$ and $\Theta_j: \cA \to \cA$ with $j=1,\ldots,d$ be the associated ideal, quotient space and shift operators associated with $(s_{\gamma})_{\gamma \in \NN_0^d}$, respectively, as in Section \ref{sec:full}. One can check that $\rank M(5) = 12$. Moreover,
$$\cI = \langle p^{(5,0)}(x,y), p^{(0,3)}(x,y) \rangle$$
satisfies $\card \cV(\cI) \leq \rank M(\infty) = 15$ and a basis for $\cC_{M(6)}$ is given by
\begin{align*}
\cB =& \; \{ 1, X, Y, X^2, XY, Y^2, X^3, X^2 Y, XY^2, X^4, X^3 Y \\
& \; \;\;  X^2 Y^2, X^4Y, X^3Y^2, X^4Y^2 \}. 
\end{align*}
Since $\cI \cap \RR[{\bf x}] = \langle \varphi(x) \rangle$ and $\cI \cap \RR[ { \bf y} ] = \langle \psi(y) \rangle$, where

$$
\varphi(x) := p^{(5,0)}(x,y)
$$
and
\begin{align*}
\psi(y) :=& \; 1650467269068800y^{15} +41261681726720000y^{14} \\
&\;+48554692101624800y^{13}-5493701075197171200y^{12}\\
& \; -28535273672045925930y^{11} +218849974650321391515y^{10} \\& \;+1568658618037258028979y^9-2053760549763388536126y^8 \\
& \;-26616254927163276895452y^7-23447661687066643695517y^6 \\
&\;+82720466156415952711103y^5+130125215370817292369264y^4\\&\;+35790490895893802856564y^3-22834780799740969878000y^2\\
& \;-9853069670890050758400y,
\end{align*}
one can check that $\gcd(\varphi, \varphi') = \gcd(\psi, \psi') = 1$ and hence $\varphi$ and $\psi$ are squarefree. Thus, we can use, e.g., Corollary 3.7.16 in \cite{KR} to obtain
$$\sqrt{\cI} = \cI + \langle {\rm sqfree}(\varphi), {\rm sqfree}(\psi) \rangle = \cI +   \langle \varphi, \psi \rangle = \cI$$
and hence $\cI$ is radical. Consequently, the shift operators $\Theta_j: \cA \to \cA$ are diagonalisable which is equivalent to the shift matrices
$$T_j =  M_{\cB}(6)^{-1} M_{\cB, \cB + {\rm e}_j}(7) \quad\quad {\rm for} \quad j=1,\ldots, d$$
being diagonalisable. Thus, Theorem \ref{thm:Feb7m1} asserts that $s$ has a $15$-atomic quasicomplex representing measure $\mu = \sum_{a=1}^{15} \varrho_a \delta_{(u_a, v_a)},$
where $$\{ (u_a, v_a) \}_{a=1}^{15} = \cV(\cI)$$
and $\varrho_1, \ldots, \varrho_{15}$ are given by \eqref{eq:WEIGHTS}.

We have already observed that
\begin{equation}
\label{eq:psi}
\psi(x) = x\left(x^4+\left(\frac{4367271}{897715}\right)-\left(\frac{12131209}{3590860}\right)x-\left(\frac{3792527}{718172}\right)x^2+x^3 \right)
\end{equation}
is square-free and hence has no repeated roots. Moreover, it is easy to check that the quartic in \eqref{eq:psi} has four real distinct zeros, say $u_1 = 0, u_2 \ldots, u_5$. One can then obtain expressions for $u_2, \ldots, u_5$ in terms of radicals and check that the cubic $p^{(0,3)}(u_a, y)$ has $3$ distinct real roots (e.g., via showing that the discriminant is positive) for $a=1,\ldots,5$. For the convenience of the reader, we have only given closed form expressions for $(u_1, v_1), (u_2, v_2)$ and $(u_3, v_3)$ and $16$ digit approximations for $\{ (u_a, v_a) \}_{a=4}^{15},$ namely
%
%
%
%
%
%
%
%
\begin{align*}
(u_1, v_1) =& \;  (0,0) \\
(u_2, v_2) =& \; \left(0, \frac{-13+\sqrt{889}}{8} \right) \\
(u_3, v_3) =& \; \left( 0, \frac{-13-\sqrt{889}}{8} \right) \\
(u_4, v_4) \approx& \; (0.7624922873183873, 0.4339289106669871) \\
(u_5, v_5) \approx& \; (0.7624922873183873, 5.285799329792782) \\
(u_6, v_6) \approx& \; ( 0.7624922873183873, -2.297920726423880 ) \\
(u_7, v_7) \approx& \; (1.945840648481577, 6.400158217920288 ), \\
 (u_8,v_8) \approx& \; (1.945840648481577,8.202371056743452) \\
 (u_9, v_9) \approx& \; (1.945840648481577, -0.8264236004499402 )  \\
 (u_{10}, v_{10}) \approx& \; (-1.455383893896135, -0.6507966134719535) \\
 (u_{11}, v_{11}) \approx& \; (-1.455383893896135, -4.542562147613256  ) \\
 (u_{12}, v_{12}) =& \; (-1.455383893896135, -10.79125031050597) \\
 (u_{13}, v_{13}) =& \; (-2.252949041903829, -0.5871708139396136) \\
 (u_{14}, v_{14}) =& \; (-2.252949041903829, -10.05498263445494) \\
 (u_{15}, v_{15}) =& \; (-2.252949041903829, -12.32115066826395)
\end{align*}
in which case we can use \eqref{eq:WEIGHTS} to obtain $16$-digit approximations for the weights $\varrho_1, \ldots, \varrho_{15}$ given by
\begin{align*}                                               
   \varrho_1  \approx& \; 1.1485082102759 \\                         
     \varrho_2  \approx& \;   0.67540495615813 \\                                                  
     \varrho_3  \approx& \;    0.13925685438668 \\
        \varrho_4  \approx& \;  8.69674628194238\\
        \varrho_5  \approx& \;  -0.27487370942507 \\
        \varrho_6  \approx& \;  -0.31699335231750\\
       \varrho_7  \approx& \;   0.45867407020225\\
        \varrho_8 \approx& \;   0.78731658438486 \\
        \varrho_9  \approx& \;    0.013888883967842 \\
        \varrho_{10}  \approx& \;  0.21808644955274 \\                                                 
       \varrho_{11}  \approx& \;   0.96056711695518 \\
       \varrho_{12}  \approx& \;   -.06450309890826 \\
       \varrho_{13}  \approx& \;  -0.013834083055744 \\                                                 
       \varrho_{14}  \approx& \;   0.74414479745023 \\
       \varrho_{15}  \approx& \;   0.827610038430212 
\end{align*}

\end{example}

%
%
%
%
%
%
%
%

%
%

\appendix
\section{Background on Pontryagin spaces}
\label{app:P}
Let $\kappa$ be a nonnegative integer. A {\it pre-Pontryagin space with $\kappa$-negative squares} is a complex (or real) vector space $\mathcal{V}$ together with a complex-valued (or real-valued) Hermitian form $( \cdot, \cdot )$ on $\mathcal{V} \times \mathcal{V}$ that satisfies the following properties:
\begin{enumerate}
\item[(i)] There exists a {\it negative subspace} $\mathcal{N} \subseteq \mathcal{V}$ with $\dim \mathcal{N} = \kappa$, i.e., there exists a $\kappa$-dimensional subspace $\mathcal{N} \subseteq \mathcal{V}$ such that
$$( v, v ) < 0 \quad {\rm for} \quad v \in \mathcal{N}.$$
\item[(ii)] $\cN$ is maximal in the sense that any other negative subspace $\cM\subseteq \cV$ satisfies 
$$\dim \cM \leq \kappa.$$
\end{enumerate}
It is readily checked that if $\cN$ is a negative subspace with dimension $\kappa$, then 
$$\cP := \cN^{\perp} = \{ v \in \cV: \text{$(v, w) = 0$ for all $w \in \cN$} \}$$
is a pre-Hilbert space with respect to the Hermitian form $( \cdot , \cdot )$, i.e.,
\begin{align*}
(v, w) =& \; \overline{(w, v) } \\
(c \, v, w) =& \; c (v, w) \\
\tag*{and} \\
(v+w, u) =& \; (v, u) + (w, u) \quad {\rm for} \quad \text{$c \in \CC \; {\rm ( or} \;  \RR {\rm )}$ and $u, v, w \in \cV.$}
\end{align*}

If we consider the Hilbert space completion $\overline{\cP}$ of $\cP$, then 
\begin{equation}
\label{eq:Oct27gn3}
\Pi_{\kappa} = \overline{\cP} \oplus \cN
\end{equation}
is a {\it Pontryagin space with $\kappa$-negative squares}. The decomposition \eqref{eq:Oct27gn3} is, of course, not unique.

It turns out that the definition of $\Pi_{\kappa}$ in \eqref{eq:Oct27gn3} does not depend on the choice of $\cN$. In view of \eqref{eq:Oct27gn3}, every $v \in \Pi_{\kappa}$ can be written as $v_+ + v_-$, where $v_+ \in \overline{\cP}$ and $v_- \in \cN$. If we let,
\begin{equation}
\label{eq:Oct27hj3}
[ v, w ] = (v_+, w_+) - (v_-, w_-) \quad {\rm for} \quad v, w \in \Pi_{\kappa},
\end{equation}
then $[ \cdot, \cdot ]$ is a usual inner product. If $v = w$, then \eqref{eq:Oct27hj3} can be written as
\begin{equation}
\label{eq:Oct27hj4}
[ v, v] = (v, v) - 2 (v_-, v_-), \quad {\rm for} \quad v \in \Pi_{\kappa}.
\end{equation}
Thus, $\Pi_{\kappa}$ may also be regarded as a Hilbert space with norm 
\begin{equation}
\label{eq:Oct27ew2}
\| v \| := [ v, v ] = \sqrt{ (v_+, v_+) - (v_-, v_-) }.
\end{equation} 
It turns out that all the norms corresponding different decompositions in \eqref{eq:Oct27gn3} are equivalent. Thus, we may mention continuity and convergence in $\Pi_{\kappa}$ without referencing a particular norm. It can be readiy checked that an analogue of the Cauchy-Schwarz inequality hold, i.e., 
\begin{equation}
\label{eq:Oct27uu3}
|( v, w) | \leq \| v \| \| w \| \quad {\rm for} \quad v, w \in \Pi_{\kappa}.
\end{equation}
It follows from \eqref{eq:Oct27uu3} that $(\cdot, \cdot)$ is continuous the first and second argument.

\begin{defn}
\label{def:June20d1}
Let $\Pi_{\kappa}$ and $\wt{\Pi}_{\kappa'}$ be Pontryagin spaces with $\kappa$ and $\kappa'$ negative squares, respectively. If $( \cdot, \cdot)_{\Pi_{\kappa}}$ and $(\cdot, \cdot)_{\Pi_{\kappa}'}$ are the Hermitian forms associated with $\Pi_{\kappa}$ and $\wt{\Pi}_{\kappa'}$, respectively, then we will call  $\Pi_{\kappa}$ and $\wt{\Pi}_{\kappa'}$ {\it isomorphic} if there exists a surjective linear map $U: \Pi_{\kappa} \to \wt{\Pi}_{\kappa'}$ such that
\begin{equation}
\label{eq:June20e1}
( v, w )_{ \Pi_{\kappa}} = (Uv ,Uw)_{\Pi_{\kappa'}}.
\end{equation}
\end{defn}

\begin{rem}
\label{rem:June20r1}
It follows immediately from \eqref{eq:June20e1} that if $\Pi_{\kappa}$ and $\wt{\Pi}_{\kappa'}$ are isomorphic, then $\kappa = \kappa'$. Indeed, if we write 
$$\Pi_{\kappa} = \cP \oplus \cN$$ 
and 
$$\wt{\Pi}_{\kappa'} = \wt{\cP} \oplus \wt{\cN},$$
then
$v \in \cN \Longleftrightarrow U v \in \wt{\cN}$. If, in addition, $\Pi_{\kappa}$ and $\wt{\Pi}_{\kappa'}$ are finite dimensional, i.e., $\dim \cP < \infty$ and $\dim \wt{\cP}< \infty$, then 
\begin{equation}
\label{eq:June20p1}
\dim \Pi_{\kappa} = \dim \Pi_{\kappa'} \Longleftrightarrow \dim \cP = \dim \wt{\cP}.
\end{equation}
\end{rem}

\begin{thm}
\label{thm:Pontryagin1}
Let $\Pi_{\kappa}$ be a finite dimensional Pontryagin space with $\kappa$-negative squares endowed with the Hermitian form $( \cdot , \cdot)$ and put $r = \dim \Pi_k$. Then there exists a unique invertible Hermitian matrix $H \in \CC^{r \times r}$ such that 
\begin{equation}
\label{eq:June20RRT}
(v, w) = \langle Hv, w \rangle,
\end{equation} 
where $\langle \cdot, \cdot \rangle$ denotes the usual inner product on $\CC^r$.
\end{thm}

\begin{proof}
See, e.g., \cite[p.~8]{GLR}
\end{proof}

\bibliographystyle{amsplain}
\bibliography{Bib}

\end{document}